\documentclass[letterpaper,12pt]{article}
\usepackage{amsmath,amsthm,amssymb}
\usepackage{mathtools}
\usepackage{float}
\usepackage{indentfirst}
\usepackage[margin=1in,letterpaper]{geometry}
\usepackage{cite} 
\usepackage[final]{hyperref} 
\hypersetup{
	colorlinks=true,
	linkcolor=blue,
	citecolor=blue,
	filecolor=magenta,
	urlcolor=blue         
}
\usepackage[english]{babel}
\usepackage{comment}
\usepackage{enumitem}
\usepackage{changepage}
\usepackage{authblk}
\usepackage[small]{titlesec}

\titleformat*{\section}{\normalsize\bfseries}
\titleformat*{\subsection}{\normalsize\bfseries}

\usepackage{tikz}
\usetikzlibrary{matrix, arrows.meta, positioning, calc}
\usetikzlibrary{decorations.pathreplacing}
\usepackage{mathdots}
\usepackage{nicematrix}

\newcommand{\RR}{\mathbb{R}}
\newcommand{\CC}{\mathbb{C}}
\newcommand{\NN}{\mathbb{N}}

\newcommand{\nrank}{\operatorname{nrank}}

\newcommand{\JJ}{\mathcal{J}}
\newcommand{\LL}{\mathcal{L}}
\newcommand{\OO}{\mathcal{O}}
\newcommand{\PP}{\mathcal{P}}
\newcommand{\MM}{\mathcal{M}}
\newcommand{\MN}{\mathcal{N}}
\newcommand{\KK}{\mathcal{K}}
\newcommand{\CO}{\overline{\mathcal{O}}}

\newcommand{\dH}{\mathrm{\mathbf{edH}}}
\newcommand{\dHI}{\mathrm{\mathbf{edH^I}}}

\theoremstyle{plain}
\newtheorem{theorem}{Theorem}[section]

\newtheorem{lemma}[theorem]{Lemma}
\newtheorem{proposition}[theorem]{Proposition}
\newtheorem{observation}[theorem]{Observation}

\theoremstyle{definition}

\newtheorem{example}[theorem]{Example}

\theoremstyle{remark}
\newtheorem{remark}[theorem]{Remark}

\title{\Large On equivalence classes of dissipative Hamiltonian pencils}
\author[]{\normalsize Maria Dronka}
\affil[]{\footnotesize Instytut Matematyki, Wydzia\l{} Matematyki i Informatyki, Uniwersytet Jagiello\'nski, \par
ul. \L{}ojasiewicza 6, 30-348 Krak\'ow, Poland (\href{mailto:maria.dronka@doctoral.uj.edu.pl}{maria.dronka@doctoral.uj.edu.pl})}

\date{}

\begin{document}

\maketitle

\begin{abstract}
    We study matrix pencils of the form \(\lambda E - (J-R)Q\), where \(Q^*E = E^*Q \geq 0\), \(J^* = -J\), \(R^* = R \geq 0\), and \(\lambda E - Q\) is regular, in the framework of Pokrzywa's ordering of strict equivalence orbits. We characterise the maximal elements, analyse some properties of the orbit closures, and derive a restriction on possible degenerations when \(Q = I\). We also consider the case where \(\lambda E - Q\) may be singular, which naturally arises in the study of orbit closures, and obtain a partial result on the Kronecker canonical form. The theory is illustrated by numerous examples. 
\end{abstract}

{\footnotesize
\textbf{Keywords:}
dissipative Hamiltonian system, matrix pencil, singular pencil, strict equivalence, orbit closure, Kronecker canonical form

\vspace{0.5em}

\textbf{AMS Subject Classification (2020):}
15A18, 15A21, 15A22
}

\section{Introduction}

By a \textit{matrix pencil}, we mean a first-degree matrix polynomial \(\PP = \lambda A - B\), where \(A, B \in \CC^{m \times n}\); the set of all matrix pencils with coefficients in $\mathbb C^{m \times n}$ is denoted by \(\CC[\lambda]^{m \times n}\). We say that matrix pencils \(\MM, \PP \in \CC[\lambda]^{m \times n}\) are \textit{strictly equivalent} if and only if there exist nonsingular matrices \(V \in \CC^{m \times m}\) and \(W \in \CC^{n \times n}\) such that \(\MM = V \PP W\). The \textit{strict equivalence orbit} (or simply \textit{orbit}) of a pencil \(\PP \in \CC[\lambda]^{m \times n}\) is the set of all pencils strictly equivalent to \(\PP\), that is
\[
\OO (\PP) := \{ V \PP W : V \in \CC^{m \times m} \text{ and } W \in \CC^{n \times n} \text{ are nonsingular}\}.
\]
The closure of the orbit, denoted by \(\CO\), is considered in the Euclidean topology of the set \(\CC^{m \times n} \times \CC^{m \times n}\).

The study of inclusion relations between orbit closures plays a central role in perturbation theory of matrix pencils; see e.g. \cite{ edelman_1, edelman_2, johansson}. If pencils \(\MM\) and \(\MN\) satisfy \(\CO(\MM) \subset \CO(\MN)\), then \(\MN\) is said to be \textit{more generic} than \(\MM\), or equivalently, \(\MN\) \textit{degenerates} into \(\MM\). Intuitively, an arbitrarily small perturbation of \(\MM\) may yield a pencil strictly equivalent to \(\MN\). The characterisation of the orbit closure hierarchy, minimal degenerations, and most generic canonical forms of \textit{structured} matrix pencils -- that is, pencils with structured coefficient matrices, such as (skew-)Hermitian or (skew-)symmetric ones --  remains an active area of research; see \cite{das, dmyt_2014, de teran_0, de teran_1, dmyt_2018}. A useful tool for this analysis was introduced by Pokrzywa \cite{Pokrzywa}, who defined six types of simple degenerations between blocks in the Kronecker canonical form that suffice to characterise the inclusion relation between orbit closures of matrix pencils; see Theorem \ref{th_pokrzywa} and \cite{futorny} for a modern treatment.

In this paper we employ Pokrzywa's framework to study the orbit closure structure of pencils of the form
\begin{align} \label{eq_form_general}
    \MM = \lambda E - (J-R)Q, \quad \text{with } \ &Q^*E = E^*Q \geq 0, \quad  J^* = -J, \quad
    R^* = R \geq 0,
\end{align}
where \(E, J, R, Q \in \CC^{n \times n}\), the symbol \(^*\) denotes the conjugate transpose, and by \(A \geq 0\) we denote that Hermitian matrix \(A\) is positive semidefinite. Matrix pencils of this form are associated with linear time-invariant dissipative Hamiltonian descriptor systems, a special class of port-Hamiltonian descriptor systems arising in energy-based modelling of dynamical systems; for more details and applications see e.g. \cite{beattie_1, beattie_2, blazhko, gernandt, MMW_2021, mehrmann, van der schaft}.
Apparently, the structural constraints \eqref{eq_form_general} impose a certain structure on the orbit degeneracy, as will be seen in the course of the current paper.

The main motivation for the present work is the important study \cite{MMW_2018}, which developed the theory of dissipative Hamiltonian descriptor systems; see also \cite{faulwasser, MMW_2022} for further results. The authors characterised the spectral properties of pencils of the form \eqref{eq_form_general} under the additional assumption that \(\lambda E - Q\) is regular. 
Yet, several important aspects remain unresolved.  Throughout the paper, we refer to the pencils satisfying \eqref{eq_form_general} together with regularity of \(\lambda E - Q\) as \textit{dissipative Hamiltonian (dH) pencils}.
The additional regularity assumption ensures that the eigenvalues of dH pencils are contained in the closed left half-plane; see Theorem \ref{th_2classes}. However, notice that the dH structure is generally not preserved under arbitrarily small perturbations. In particular, a dH pencil may belong to the orbit closure of a pencil of the form \eqref{eq_form_general} with singular \(\lambda E - Q\), as illustrated in the following example. 

\begin{example} \label{ex_motivation}
    Consider the following matrix pencils:
    \[
    \MM =
    \begin{bmatrix}
        \lambda & 0 \\
        0 & 0
    \end{bmatrix}
    \quad \text{and} \quad
    \MN =
    \begin{bmatrix}
        \lambda & 0 \\
        \epsilon & 0
    \end{bmatrix},
    \]
    where \(\epsilon > 0\).
    We have \(\MM \in \CO(\MN)\), so an arbitrarily small perturbation of pencil \(\MM\) may result in a pencil strictly equivalent to \(\MN\). Observe that \(\MM\) is a dH pencil, as \(\MM = \lambda E - (J-R)Q\) for
    \[
    E = \begin{bmatrix}
        1 & 0 \\
        0 & 0
    \end{bmatrix},
    \quad
    J = R = 0,
    \quad \text{and} \quad
    Q = I_2,
    \]
    with \(\lambda E - Q\) being regular. On the other hand, by Theorem \ref{th_2classes}, \(\MN\) is not a dH pencil, as it has a left minimal index one (see Example \ref{ex_kcf_def}). However, we obtain that \( \MN = \lambda \hat{E} - (\hat{J}-\hat{R})\hat{Q}\) for
    \[
    \hat{E} = \begin{bmatrix}
        1 & 0 \\
        0 & 0
    \end{bmatrix},
    \quad
    \hat{J} = \begin{bmatrix}
        0 & 1 \\
        -1 & 0
    \end{bmatrix},
    \quad
    \hat{R} = 0,
    \quad
    \text{and} \quad
    \hat{Q} = \begin{bmatrix}
        \epsilon & 0 \\
        0 & 0
    \end{bmatrix},
    \]
    where pencil \(\lambda \hat{E} - \hat{Q}\) is singular.
\end{example}

Consequently, a deeper understanding of the orbit closure structure of dH pencils and the spectral properties of pencils of the form \eqref{eq_form_general} without any regularity assumptions is essential for a comprehensive algebraic theory of dH pencils.
The present work constitutes a first step in this direction.

The main contributions of the paper are as follows. We derive a restriction on possible degenerations of dH pencils when \(Q = I\); see Theorem \ref{th_nonQ}, characterise the maximal elements of the partial order induced by orbit closure inclusion; see Theorem \ref{th_max_el}, and analyse some of their properties. In Theorem \ref{th_nondegen} we characterise
when the orbit closure of a dH pencil contains only matrix pencils strictly equivalent to dH pencils, and we illustrate it with the orbit stratification in the \(2\times2\) case. Furthermore, Theorem \ref{th_closure} is a version of the classical result stating that the closure of an orbit is the union of the orbit itself and orbits of strictly smaller dimension, restricted to the class of matrix pencils strictly equivalent to dH pencils. In all results we consider the additional case of \(Q = I\).

In addition, following \cite{MMW_2018}, we derive a complete characterisation of the Kronecker canonical form for pencils \(\lambda E - Q\) satisfying \(E^*Q = Q^*E \geq 0\); see Theorem \ref{th_EQ}. We also obtain a property of the associated orbit closures and illustrate it with the orbit stratification in the \(2 \times 3\) case.

The final result of the paper is Theorem \ref{th_singular}, which provides a partial
characterisation of the Kronecker canonical form for pencils of the form \eqref{eq_form_general} with possibly singular \(\lambda E - Q\). We show that such pencils may contain canonical blocks of all types and sizes, although not in arbitrary combinations. Moreover, for an arbitrary right singular, Jordan, or nilpotent
block, we explicitly construct a pencil satisfying
\eqref{eq_form_general} whose Kronecker canonical form contains such block; see Lemmas \ref{lem_right}, \ref{lem_jordan}, and \ref{lem_nilpotent}.

The paper is organised as follows. In Section \ref{sec_preliminaries} we introduce notation and recall preliminary results. In Section \ref{sec_dh} we study the orbit hierarchy of dH pencils, i.e. pencils of the form \eqref{eq_form_general} with regular \(\lambda E-Q\). In Section \ref{sec_EQ} we consider possibly non-square and possibly singular pencils of the form \(\lambda E-Q\) satisfying \(E^*Q = Q^*E \geq 0\). In Section \ref{sec_singular} we study pencils of the form \eqref{eq_form_general} with possibly singular \(\lambda E- Q\).

\section{Preliminaries} \label{sec_preliminaries}

Pencils belonging to the same orbit share the \textit{Kronecker canonical form} (KCF), which is a generalisation of the Jordan canonical form for matrices. To introduce it, let us first define the following four special types of matrix pencils:
\begin{itemize}
    \item \textit{right singular blocks} of size \(k \times (k + 1)\), \(k \geq 0\),
    \[
        \LL_k := \begin{bmatrix}
        \lambda & -1 & &  \\
        & \ddots & \ddots & \\
        & & \lambda & -1
        \end{bmatrix},
    \]
    \item \textit{left singular blocks}, \(\LL_k^T\), of size \((k + 1) \times k\), \(k \geq 0\),
    \item \textit{Jordan blocks} of size \(k \times k \), \(k \geq 1\), corresponding to a finite eigenvalue \(\mu \in \CC\),
    \[
        \JJ_k^\mu :=
        \begin{bmatrix}
        \lambda - \mu & -1 & &  \\
        & \lambda - \mu & \ddots & \\
        & & \ddots & -1 \\
        & & & \lambda - \mu
        \end{bmatrix},
    \]
    \item \textit{nilpotent blocks} of size \(k \times k\), \(k \geq 1\), corresponding to the eigenvalue infinity,
    \[
        \JJ_k^\infty := 
        \begin{bmatrix}
        -1 & \lambda & &  \\
        & -1 & \ddots & \\
        & & \ddots & \lambda \\
        & & & -1
        \end{bmatrix},
    \]
\end{itemize}

Note that all unspecified entries of the pencils and matrices throughout the paper are zero. For brevity, we sometimes denote the direct sum of \(p\) blocks of the same type and size, for example \(\LL_0\), by \(p \LL_0\). Let \(\overline{\CC} := \CC \cup \{\infty\}\).

\begin{theorem}[Kronecker canonical form, {\cite[Ch. XII]{Gant}}]
     Let \(\PP \in \CC[\lambda]^{m \times n}\) be a matrix pencil. Then there exist nonsingular matrices \(V \in \CC^{m \times m}\) and \(W \in \CC^{n \times n}\) such that
     \[
     V \PP W = \LL_{\alpha_1} \oplus \ldots \oplus \LL_{\alpha_p} \oplus \LL_{\beta_1}^T \oplus \ldots \oplus \LL_{\beta_q}^T \oplus \JJ_{\gamma_1}^{\mu_1} \oplus \ldots \oplus \JJ_{\gamma_s}^{\mu_{s}},
     \]
     with \(p, q, s \geq 0\), and \(\mu_j \in \overline{\CC}\). Moreover, the direct sum is uniquely determined up to permutation of the blocks.
\end{theorem}

The sizes of the singular blocks, \(\alpha_1, \ldots, \alpha_p\), and \(\beta_1, \ldots, \beta_q\), are called the \textit{right} and \textit{left minimal indices} of pencil \(\PP\), respectively.
The finite values among the \(\mu_j\)'s are called the \textit{finite eigenvalues} of pencil \(\PP\), whereas if some \(\mu_k = \infty\) we say that pencil \(\PP\) has the \textit{eigenvalue infinity}. 
The sizes \(\gamma_i\) of blocks associated with a fixed eigenvalue \(\mu \in \overline{\CC}\) are called the \textit{partial multiplicities} of \(\mu\). 
An eigenvalue is \textit{semisimple} if its partial multiplicities are all equal to one.
The \textit{index} of pencil \(\PP\) is the size of the largest nilpotent block associated with the eigenvalue infinity. If the KCF of pencil \(\PP\) does not contain any such blocks, i.e. infinity is not an eigenvalue of \(\PP\), the index of pencil \(\PP\) is equal to zero.

\begin{example} \label{ex_kcf_def}
    Consider the matrix pencils from Example \ref{ex_motivation}. We obtain
    \[
    \MM = \begin{bmatrix}
        \lambda & 0 \\
        0 & 0
    \end{bmatrix}
    = \JJ_1^0 \oplus \LL_0 \oplus \LL_0^T,
    \quad \text{and} \quad
    \begin{bmatrix}
        1 & 0 \\
        0 & - \frac{1}{\epsilon}
    \end{bmatrix}
    \MN =
    \begin{bmatrix}
        \lambda & 0 \\
        - 1 & 0
    \end{bmatrix}
    = \LL_1^T \oplus \LL_0.
    \]
    Therefore, \(\MM\) has a right minimal index zero, a left minimal index zero, and a semisimple eigenvalue \(\mu = 0\), as the partial multiplicity of \(\mu\) is one. Pencil \(\MN\) has a right minimal index zero, a left minimal index one, and no eigenvalues.
\end{example}

The \textit{normal rank} (or shortly \textit{rank}) of pencil \(\PP\) is the size of its largest non-identically zero minor and is denoted by \(\nrank \PP\). A matrix pencil is called \textit{regular} if it is square and of the full rank, or equivalently, if it has no minimal indices. Otherwise, it is called \textit{singular}. We divide the KCF of a pencil into the \textit{regular part}, which is the direct sum of all Jordan and nilpotent blocks, and the \textit{singular part}, which is a direct sum of all singular blocks.

We say that matrices \(A, B \in \CC^{m \times n}\) have a \textit{common right nullspace} if \(\ker A \ \cap \ \ker B \neq \{0\} \).

The Kronecker canonical form of dH pencils was characterised in \cite[Th. 3.1]{MMW_2022}, with the additional case of \(Q = I\), where \(I\) denotes the identity matrix of proper size. We recall this result below, as we will often refer to it throughout the paper.

\begin{theorem}
\label{th_2classes}
    \begin{enumerate}[label=\rm(\roman*), nosep]
    \item \label{dH_class} A pencil \(\MM \in \CC[\lambda]^{n \times n}\) is strictly equivalent to a pencil of the form \eqref{eq_form_general} with regular \(\lambda E - Q\) if and only if the following conditions hold:
        \begin{enumerate}[label=\rm(\alph*), nosep]
        \item The eigenvalues of \(\MM\) lie in the closed left half-plane.
        \item The finite nonzero eigenvalues on the imaginary axis are semisimple, and the partial multiplicities of the eigenvalue zero are at most two.
        \item The index is at most two.
        \item The left minimal indices are all zero, and the right minimal indices are at most one (if there are any).
        \end{enumerate}

    \item \label{dHI_class} A pencil \(\MM \in \CC[\lambda]^{n \times n}\) is strictly equivalent to a pencil of the form
    \eqref{eq_form_general} with \(Q = I\) if and only if the following conditions hold:
        \begin{enumerate}[label=\rm(\alph*), nosep]
        \item The eigenvalues of \(\MM\) lie in the closed left half-plane.
        \item The finite eigenvalues on the imaginary axis (including zero) are semisimple.
        \item The index is at most two.
        \item The left and right minimal indices are all zero (if there are any).
        \end{enumerate}
    \end{enumerate}
\end{theorem}

We will refer to the sets of matrix pencils satisfying the conditions described in \ref{dH_class} and \ref{dHI_class} as classes \(\dH\) and \(\dHI\), respectively. The notation reflects the fact that class \(\dH\) consists of pencils strictly equivalent to dH pencils (in particular, it contains all dH pencils), whereas \(\dHI\) consists of pencils strictly equivalent to dH pencils satisfying the additional assumption \(Q = I\). Notice that \(\dHI \subset \dH\) and the defining conditions differ only in the partial multiplicities of the eigenvalue zero and the right minimal indices.

Our main tool for investigating the orbit closure structure of matrix pencils will be the following result, originating in the work of Pokrzywa \cite{Pokrzywa}, derived in \cite{boley}, and presented here in the form of \cite[Th. 2.2]{dmyt_2018}.

\begin{theorem} \label{th_pokrzywa}
    Let \(\MM, \PP \in \CC[\lambda]^{m \times n}\) be two matrix pencils. Then \(\PP \in \CO(\MM)\) if and only if the KCF of \(\MM\) can be obtained  from the KCF of \(\PP\) through a finite sequence of operations on the blocks of the following six types:
    
    \begin{enumerate}[itemsep=-3pt, topsep=0pt]
        \item\label{R1} \(\LL_{j-1} \oplus \LL_{k+1} \rightsquigarrow \LL_j \oplus \LL_k, \: 1 \leq j \leq k\);
        \item\label{R2} \(\LL_{j-1}^T \oplus \LL_{k+1}^T \rightsquigarrow \LL_j^T \oplus \LL_k^T, \: 1 \leq j \leq k\);
        \item\label{R3} \(\LL_j \oplus \JJ_{k+1}^\mu \rightsquigarrow \LL_{j+1} \oplus \JJ_k^\mu, \: j, k \geq 0 \text{ and } \mu \in \overline{\CC}\);
        \item\label{R4} \(\LL_j^T \oplus \JJ_{k+1}^\mu \rightsquigarrow \LL_{j+1}^T \oplus \JJ_k^\mu, \: j, k \geq 0 \text{ and } \mu \in \overline{\CC}\);
        \item\label{R5} \(\JJ_j^\mu \oplus \JJ_k^\mu \rightsquigarrow \JJ_{j-1}^\mu \oplus \JJ_{k+1}^\mu, \: 1 \leq j \leq k \text{ and } \mu \in \overline{\CC}\);
        \item\label{R6} \(\LL_p \oplus \LL_q^T \rightsquigarrow \bigoplus_{i=1}^t \JJ_{k_i}^{\mu_i}, \text{ if } p + q + 1 = \sum_{i=1}^t k_i, \text{ and } \mu_i \neq \mu_{j} \text{ for } i \neq j, \mu_i \in \overline{\CC}\).
    \end{enumerate}
\end{theorem}

Here, the symbol \(\rightsquigarrow\) indicates that, in the KCF of a matrix pencil, the blocks on the left-hand side are replaced by those on the right-hand side. Moreover, by \(\JJ_0^\mu\) we denote the \(0 \times 0\) matrix.

\section{Strict equivalence orbit structure of dissipative Hamiltonian pencils} \label{sec_dh}

\subsection{Restriction on degenerations in the case \(Q = I\)}

Let us begin with an observation that a pencil from class \(\dHI\) cannot degenerate into a pencil outside this class without a change in rank, i.e.

\begin{theorem} \label{th_nonQ}
    Let \(\MM \in \CC[\lambda]^{n \times n}\) be a matrix pencil belonging to \(\dHI\). If pencil \(\PP \in \CC[\lambda]^{n \times n}\) satisfies \(\PP \in \CO(\MM)\) and \(\nrank \PP = \nrank \MM\), then \(\PP\) also belongs to \(\dHI\).    
\end{theorem}

\begin{proof}
Let \(\MM, \PP \in \CC[\lambda]^{n \times n}\) be matrix pencils satisfying the assumptions of the theorem. Without loss of generality, we may assume that all pencils are in the KCF.

By Theorem \ref{th_pokrzywa}, there exists a finite sequence of matrix pencils \((\PP_i)_{i=0}^m\) such that \hbox{\(\PP_0 = \PP\)}, \(\PP_m = \MM\), and \(\PP_i\) can be transformed into \(\PP_{i+1}\) by applying one of the operations \ref{R1}-\ref{R6}, for \(0 \leq i \leq m - 1\). Notice that operations \ref{R1}-\ref{R5} do not change the rank of the pencil, while operation \ref{R6} increases it. Hence, from the equality of ranks of \(\PP\) and \(\MM\), we must have
\[
\nrank \PP = \nrank \PP_1 = \nrank \PP_2 = \ldots = \nrank \MM,
\]
and operation \ref{R6} can be excluded.

To complete the proof, it suffices to show that for \(1 \leq i \leq m\), if \(\PP_i\) belongs to class \(\dHI\), and \(\PP_i\) is obtained from \(\PP_{i-1}\) via one of the operations \ref{R1}-\ref{R5}, then \(\PP_{i-1}\) also belongs to class \(\dHI\).

First, notice that, by Theorem \ref{th_2classes}, such \(\PP_i\) contains only the singular blocks of size \(0\) (if any). Thus, we can exclude operations \ref{R1}-\ref{R4}, as they would require \(\PP_i\) to contain singular blocks of size at least \(1\). It follows that \(\PP_i\) must be obtained from \(\PP_{i-1}\) by applying operation \ref{R5}. Let us check that in such a case \(\PP_i\) must satisfy all the properties of matrix pencils belonging to class \(\dHI\), as described in Theorem \ref{th_2classes}:

\begin{enumerate}[label=(\alph*)]
    
    \item Operation \ref{R5} does not influence the set of eigenvalues of the pencil, i.e., the set of eigenvalues of \(\PP_{i-1}\) is the same as that of \(\PP_i\), which is contained in the closed left half-plane.
    
    \item Let \(\mu\) be an arbitrary finite eigenvalue of the pencil \(\PP_{i-1}\) lying on the imaginary axis. The Jordan blocks associated with \(\mu\) in \(\PP_i\) are of size at most \(1\). However, applying operation \ref{R5} would yield two blocks which differ in size by at least \(2\) (or an empty matrix \(\JJ_0^\mu\) and a block of size at least \(2\)). It follows that \(\PP_{i}\) cannot be obtained from \(\PP_{i-1}\) by applying operation \ref{R5} to the blocks associated with \(\mu\). Hence, the blocks of this type in \(\PP_{i-1}\) are the same as in \(\PP_i\).
    
    \item Consider the Jordan blocks associated with the eigenvalue \(\infty\). Such blocks in \(\PP_i\) are of size at most \(2\). Therefore, the only possibility of \(\PP_i\) being obtained as a result of applying operation \ref{R5} to the blocks associated with the eigenvalue \(\infty\) in \(\PP_{i-1}\) is
    \[
    \PP_{i-1} = \hat{\PP} \oplus \JJ_1^\infty \oplus \JJ_1^\infty \rightsquigarrow \hat{\PP} \oplus \JJ_0^\infty \oplus \JJ_2^\infty = \PP_i,
    \]
    where \(\hat{\PP}\) is the appropriate subpencil of \(\PP_{i-1}\) that remains unchanged. In this case, \(\PP_{i-1}\) contains the blocks associated with the eigenvalue \(\infty\) that also appear in \(\PP_i\), as \(\hat{\PP}\) is also a subpencil of \(\PP_i\), along with two blocks \(\JJ_1^\infty\), all of size at most \(2\).
    Otherwise, \(\PP_{i-1}\) contains precisely the same blocks of this type as \(\PP_i\).
    
    \item Operation \ref{R5} leaves the singular blocks unchanged. Thus, \(\PP_{i-1}\) contains the same singular blocks as \(\PP_i\), i.e. singular blocks of size \(0\) (if any).
\end{enumerate}
\end{proof}

All assumptions in Theorem \ref{th_nonQ} are necessary. The necessity of the rank equality assumption is illustrated in the example below, while for the other assumptions it is straightforward.

\begin{example}
    Consider the pencils \(\MM = \JJ_2^\infty\) and \(\PP = \LL_0 \oplus \LL_1^T\), so that we have
    \[
    \nrank \MM = 2 > 1 = \nrank \PP.
    \]
    The inclusion \(\PP \in \CO(\MM)\) follows from Theorem \ref{th_pokrzywa}, since \(\PP\) can be transformed into \(\MM\) by applying operation \ref{R6}. However, by Theorem \ref{th_2classes}, \(\MM\) belongs to class \(\dHI\), while \(\PP\) does not, as it contains a left singular block of size \(1\).
\end{example}

Theorem \ref{th_nonQ} is no longer true when we consider the larger class \(\dH\), as the next example shows. This means that pencils belonging to \(\dH\) can degenerate into pencils outside this class even without a change in rank.

\begin{example}
    Consider the matrix pencils \(\MM = \LL_1 \oplus \LL_0^T\) and \(\PP = \LL_0 \oplus \LL_0^T \oplus \JJ_1^1\). Pencil \(\PP\) can be transformed into pencil \(\MM\) via operation \ref{R3}, hence \(\PP \in \CO(\MM)\) by Theorem \ref{th_pokrzywa}. Moreover, we have \(\nrank \MM =  \nrank \PP = 1\). However, by Theorem \ref{th_2classes} pencil \(\MM\) belongs to class \(\dH\), while pencil \(\PP\) does not, as it has an eigenvalue in the open right half-plane.
\end{example}

Moreover, there are no further implications between the assumptions and the conclusion of Theorem \ref{th_nonQ}, as demonstrated by the counterexamples below.

\begin{example}
    \begin{enumerate}[label=\rm(\arabic*)]
        \item Let \(\MM = \LL_0 \oplus \LL_1^T\) and \(\PP = \LL_0 \oplus \LL_0^T \oplus \JJ_1^0\). Then \(\PP \in \CO(\MM)\) by operation \ref{R4} from Theorem \ref{th_pokrzywa}, and \(\nrank \MM = \nrank \PP = 1\). However, by Theorem \ref{th_2classes}, \(\PP\) belongs to class \(\dHI\), while \(\MM\) does not.
        \item By Theorem \ref{th_2classes}, pencils \(\MM = \JJ_1^0\) and \(\PP = \LL_0 \oplus \LL_0^T\) both belong to class \(\dHI\), and \(\PP \in \CO(\MM)\) via operation \ref{R6} from Theorem \ref{th_pokrzywa}. However, \(\nrank \MM = 1 > 0 = \nrank \PP\).
        \item For \(\MM = \JJ_1^\infty\) and \(\PP = \JJ_1^0\), we have \(\nrank \MM = \nrank \PP = 1\), and both pencils belong to class \(\dHI\) by Theorem \ref{th_2classes}. However, \(\PP \notin \CO(\MM)\), as it is straightforward that \(\MM\) cannot be obtained from \(\PP\) by applying any of the operations \ref{R1}-\ref{R6} from Theorem~\ref{th_pokrzywa}.
        \end{enumerate}
\end{example}

\subsection{Maximal elements of the partial order induced by orbit closure inclusion}

For pencils \(\MM, \PP \in \CC[\lambda]^{n \times n}\), let us denote the partial order relation \(\CO(\PP) \subset \CO(\MM)\) by \(\PP \prec_c \MM\). In the following theorem, we characterise the maximal elements of the partial order \(\prec_c\) within classes \(\dH\) and \(\dHI\).

\begin{theorem} \label{th_max_el}
\begin{enumerate}[label=\rm(\roman*)]
    \item \label{max_el_i} The maximal elements of the partial order \(\prec_c\) in \(\CC[\lambda]^{n \times n}\) within class \(\dH\) are of the following KCF, up to permutation of the blocks:
    \begin{equation}\label{eq_max_dh}
    \begin{split}
    &\JJ_{s_1}^{\mu_1} 
    \oplus \ldots 
    \oplus \JJ_{s_p}^{\mu_p}
    \oplus \underbrace{\JJ_1^{\nu_1}
    \oplus \ldots
    \oplus \JJ_1^{\nu_1}}_{n_1 \text{ blocks}}
    \oplus \ldots 
    \oplus \underbrace{\JJ_1^{\nu_q}
    \oplus \ldots
    \oplus \JJ_1^{\nu_q}}_{n_q \text{ blocks}}
    \\
    &\oplus \underbrace{\JJ_1^0}_{d_0 \text{ blocks}}
    \oplus \underbrace{\JJ_2^0
    \oplus \ldots
    \oplus \JJ_2^0}_{v \text{ blocks}}
    \oplus \underbrace{\JJ_1^\infty}_{d_\infty \text{ blocks}} 
    \oplus \underbrace{\JJ_2^\infty 
    \oplus \ldots 
    \oplus \JJ_2^\infty}_{w \text{ blocks}},
    \end{split}
    \end{equation}
    where \(\mu_1, \ldots, \mu_p\) are distinct eigenvalues contained in the open left half-plane, \\
    \(\nu_1, \ldots, \nu_q\) are distinct nonzero eigenvalues lying on the imaginary axis, \\
    \(p, q, v, w \geq 0\), \(s_1, \ldots, s_p, n_1, \ldots, n_q \geq 1\), and \(0 \leq d_0, d_\infty \leq 1\), with
    \[
    \sum_{i=1}^p s_i + \sum_{j=1}^q n_j + d_0 + 2v + d_\infty + 2w = n.
    \]
    (We follow the convention that a sum with an upper index smaller than the lower index is equal to zero.)
    
    \item \label{max_el_ii} The maximal elements of the partial order \(\prec_c\) in \(\CC[\lambda]^{n \times n}\) within class \(\dHI\) are of the following KCF, up to permutation of the blocks:
    \begin{equation}\label{eq_max_dhi}
    \begin{split}
    &\JJ_{s_1}^{\mu_1} 
    \oplus \ldots 
    \oplus \JJ_{s_p}^{\mu_p}
    \oplus \underbrace{\JJ_1^{\nu_1}
    \oplus \ldots
    \oplus \JJ_1^{\nu_1}}_{n_1 \text{ blocks}}
    \oplus \ldots 
    \oplus \underbrace{\JJ_1^{\nu_q}
    \oplus \ldots
    \oplus \JJ_1^{\nu_q}}_{n_q \text{ blocks}}
    \\
    &\oplus \underbrace{\JJ_1^0
    \oplus \ldots
    \oplus \JJ_1^0}_{n_0 \text{ blocks}}
    \oplus \underbrace{\JJ_1^\infty}_{d_\infty \text{ blocks}} 
    \oplus \underbrace{\JJ_2^\infty 
    \oplus \ldots 
    \oplus \JJ_2^\infty}_{w \text{ blocks}},
    \end{split}
    \end{equation}
    where \(\mu_1, \ldots, \mu_p\) are distinct eigenvalues contained in the open left half-plane, \\
    \(\nu_1, \ldots, \nu_q\) are distinct nonzero eigenvalues lying on the imaginary axis, \\
    \(p, q, n_0, w \geq 0\), \(s_1, \ldots, s_p, n_1, \ldots, n_q \geq 1\), and \(0 \leq d_\infty \leq 1\), with
    \[
    \sum_{i=1}^p s_i + \sum_{j=0}^q n_j + d_\infty + 2w = n.
    \]
\end{enumerate}    
\end{theorem}

\begin{proof}
\begin{enumerate}[label=\rm(\roman*)]
    \item Let \(\MM\) be an arbitrary maximal element of the partial order \(\prec_c\) in \(\CC[\lambda]^{n \times n}\) within class \(\dH\). Without loss of generality, assume \(\MM\) is in the KCF. Since pencil \(\MM\) is square, it contains the same number of left and right singular blocks. Thus, if \(\MM\) contained some singular block, there would be at least two blocks, \(\LL_k\) and \(\LL_0^T\), with \(0 \leq k \leq 1\). We would then have, by applying operation \ref{R6} from Theorem \ref{th_pokrzywa},
    \[
    \MM = \hat{\MM} \oplus \LL_k \oplus \LL_0^T \prec_c \hat{\MM} \oplus \JJ_{k+1}^\mu,
    \]
    for an appropriate subpencil \(\hat{\MM}\) and an arbitrary eigenvalue \(\mu\) lying in the open left half-plane. This contradicts the maximality of \(\MM\), as the pencil on the right-hand side also belongs to class \(\dH\) by Theorem \ref{th_2classes}.
    The restriction on the number of Jordan blocks associated with the eigenvalue zero follows from the fact that if there were at least two blocks \(\JJ_1^0\), we would have for an appropriate subpencil \(\hat{\MM}\),
    \[
    \MM = \hat{\MM} \oplus \JJ_1^0 \oplus \JJ_1^0 \prec_c \hat{\MM} \oplus \JJ_2^0,
    \]
    by applying operation \ref{R5}. Again, this contradicts the maximality of \(\MM\) within class \(\dH\). If there were at least two blocks \(\JJ_1^\infty\), or two blocks \(\JJ_{s_{i}}^{\mu_i}\), \(\JJ_{\hat{s}_{i}}^{\mu_i}\) associated with a fixed eigenvalue \(\mu_i\) lying in the open left half-plane, \(s_i, \hat{s}_i \geq 1\), we can apply the analogous reasoning.

    For the sufficiency part of the proof, let \(\PP\) be a KCF satisfying \eqref{eq_max_dh}. Notice that if \(\PP\) does not contain at least two Jordan blocks associated with the same eigenvalue, we cannot apply to it any operation from Theorem \ref{th_pokrzywa}, therefore \(\PP\) is maximal. 
    If \(\PP\) does contain two blocks associated with the same eigenvalue, we can apply operation \ref{R5} and obtain pencil \(\PP'\) such that \(\PP \prec_c \PP'\). However, observe that any pencil \(\PP'\) obtained from \(\PP\) in this way does not belong to class \(\dH\) by Theorem \ref{th_2classes}.
    
    \item The proof is analogous to case \ref{max_el_i}.
    \end{enumerate}
\end{proof}

Notice that the KCFs of the maximal elements described in \eqref{eq_max_dh} and \eqref{eq_max_dhi} differ only in the Jordan blocks corresponding to the eigenvalue zero.

\begin{observation}
    The maximal elements of the partial order \(\prec_c\) within class \(\dHI\) are maximal within class \(\dH\) if and only if their KCF contains at most one Jordan block of size one associated with the eigenvalue zero, \(\JJ_1^0\).
\end{observation}

The next proposition summarises the relationship between the maximal elements within the two classes and the maximal elements in the whole set of complex \(n \times n\) matrix pencils.

\begin{proposition}
    The maximal elements of the partial order \(\prec_c\) within class \(\dH\) (or class \(\dHI\))
    are also maximal within the whole set of complex \(n \times n\) matrix pencils, \(\CC[\lambda]^{n \times n}\), if and only if their KCF is of the form
    \begin{equation} \label{max_el_CC}
    \JJ_{k_1}^{\nu_1} \oplus \ldots \oplus \JJ_{k_t}^{\nu_t},
    \end{equation}
    with \(\nu_i \neq \nu_j\) for \(i \neq j\), and \(\sum_{i = 1}^t k_t = n\).  
\end{proposition}

\begin{proof}
    Let \(\MM\) be the KCF of an arbitrary maximal element of the partial order \(\prec_c\) within class \(\dH\). The same argument applies to class \(\dHI\), replacing \(\dH\) by \(\dHI\).

    First, assume that \(\MM\) is not of the form \eqref{max_el_CC}, i.e. it contains at least two Jordan blocks associated with some eigenvalue \(\nu\). Assume \(\MM\) contains blocks \(J_p^\nu\) and \(\JJ_q^\nu\), where \(1 \leq p \leq q\). Then, by applying operation \ref{R5}, we obtain
    \[
    \MM = \hat{\MM} \oplus \JJ_p^\nu \oplus \JJ_q^\nu \prec_c \hat{\MM} \oplus \JJ_{p-1}^\nu \oplus \JJ_{q+1}^\nu,
    \]
    for an appropriate subpencil \(\hat{\MM}\). Thus, \(\MM\) is not maximal outside class \(\dH\).

    Let us now assume that \(\MM\) is of the form \eqref{max_el_CC}. Notice that in this case, we cannot apply to \(\MM\) any of the operations \ref{R1}-\ref{R6}. Therefore, by Theorem \ref{th_pokrzywa} there does not exist a pencil \(\MM' \in \CC[\lambda]^{n \times n}\) such that \(\MM \prec_c \MM'\), hence \(\MM\) is maximal in the whole set of complex \(n \times n\) matrix pencils.
\end{proof}

\subsection{Properties of the orbit closures}

Let us now consider the closures of the strict equivalence orbits of dH pencils. We begin with the analysis of the orbit stratification in the simple case of \(2 \times 2\) matrix pencils, depicted in Figure \ref{fig_strat}.

\begin{figure}
    \centering
    
        \begin{tikzpicture}

  \matrix (m) [matrix of nodes, row sep=0.25cm, column sep=2cm, nodes in empty cells]
  {
    & & &[+1cm] {\color{purple} \(\JJ_2^{\lambda_r}\)} \; \; \; \; \; \\
    & & & {\color{purple} \(\JJ_1^{\lambda_r} \oplus \JJ_1^{\hat{\lambda}_r}\)} \\
    & & {\color{purple} \(2 \JJ_1^{\lambda_r}\)} & {\color{purple} \(\JJ_1^{\lambda_r} \oplus \JJ_1^{\lambda_i}\)} \\
    & & & {\color{purple} \(\JJ_1^{\lambda_r} \oplus \JJ_1^0\)} \\
    & {\color{purple} \(\LL_0 \oplus \LL_0^T \oplus \JJ_1^{\lambda_r}\)} & \(2 \JJ_1^{\lambda_i}\) & {\color{purple} \(\JJ_1^{\lambda_r} \oplus \JJ_1^{\lambda_l}\)} \\
    & & & {\color{purple} \(\JJ_1^{\lambda_r} \oplus \JJ_1^\infty\)} \\
    & \(\LL_0 \oplus \LL_0^T \oplus \JJ_1^{\lambda_i}\) & {\color{purple} \(\LL_0 \oplus \LL_1^T\)} & {\color{purple} \(\JJ_2^{\lambda_i}\)} \; \; \; \; \; \\
    & & & \(\JJ_1^{\lambda_i} \oplus \JJ_1^{\hat{\lambda}_i}\) \\
    \(2\LL_0 \oplus 2\LL_0^T\) & \(\LL_0 \oplus \LL_0^T \oplus \JJ_1^0\) & {\color{blue} \(\LL_1 \oplus \LL_0^T\)} & \(\JJ_1^{\lambda_i} \oplus \JJ_1^0\)  \\
    & & & \(\JJ_1^{\lambda_i} \oplus \JJ_1^{\lambda_l}\) \\
    & \(\LL_0 \oplus \LL_0^T \oplus \JJ_1^{\lambda_l}\) & \(2 \JJ_1^0\) & \(\JJ_1^{\lambda_i} \oplus \JJ_1^\infty\) \\
    & & & {\color{blue} \(\JJ_2^0\)} \; \; \; \; \; \\
    & \(\LL_0 \oplus \LL_0^T \oplus \JJ_1^\infty\) & \(2 \JJ_1^{\lambda_l}\) & \(\JJ_1^0 \oplus \JJ_1^{\lambda_l}\) \\
    & & & \(\JJ_1^0 \oplus \JJ_1^\infty\) \\
    & & \(2 \JJ_1^\infty\) & \(\JJ_2^{\lambda_l}\) \; \; \; \; \; \\
    & & & \(\JJ_1^{\lambda_l} \oplus \JJ_1^{\hat{\lambda}_l}\) \\
    & & & \(\JJ_1^{\lambda_l} \oplus \JJ_1^\infty\) \\
    & & & \(\JJ_2^\infty\) \; \; \; \; \; \\
  };

  \draw[->] (m-9-1.east) -- (m-5-2.west);
  \draw[->] (m-9-1.east) -- (m-7-2.west);
  \draw[->] (m-9-1.east) -- (m-9-2.west);
  \draw[->] (m-9-1.east) -- (m-11-2.west);
  \draw[->] (m-9-1.east) -- (m-13-2.west);
  
  \draw[->] (m-5-2.east) -- (m-3-3.west);
  \draw[->] (m-5-2.east) -- (m-7-3.west);
  \draw[->] (m-5-2.east) -- (m-9-3.west);
  
  \draw[->] (m-7-2.east) -- (m-5-3.west);
  \draw[->] (m-7-2.east) -- (m-7-3.west);
  \draw[->] (m-7-2.east) -- (m-9-3.west);
  
  \draw[->] (m-9-2.east) -- (m-7-3.west);
  \draw[->] (m-9-2.east) -- (m-9-3.west);
  \draw[->] (m-9-2.east) -- (m-11-3.west);
  
  \draw[->] (m-11-2.east) -- (m-7-3.west);
  \draw[->] (m-11-2.east) -- (m-9-3.west);
  \draw[->] (m-11-2.east) -- (m-13-3.west);
  
  \draw[->] (m-13-2.east) -- (m-7-3.west);
  \draw[->] (m-13-2.east) -- (m-9-3.west);
  \draw[->] (m-13-2.east) -- (m-15-3.west);
  
  \draw[->] (m-3-3.east) -- (m-1-4.west);
  
  \draw[->] (m-5-3.east) -- (m-7-4.west);
  
  \draw[->] (m-7-3.east) -- (m-1-4.west);
  \draw[->] (m-7-3.east) -- (m-2-4.west);
  \draw[->] (m-7-3.east) -- (m-3-4.west);
  \draw[->] (m-7-3.east) -- (m-4-4.west);
  \draw[->] (m-7-3.east) -- (m-5-4.west);
  \draw[->] (m-7-3.east) -- (m-6-4.west);
  \draw[->] (m-7-3.east) -- (m-7-4.west);
  \draw[->] (m-7-3.east) -- (m-8-4.west);
  \draw[->] (m-7-3.east) -- (m-9-4.west);
  \draw[->] (m-7-3.east) -- (m-10-4.west);
  \draw[->] (m-7-3.east) -- (m-11-4.west);
  \draw[->] (m-7-3.east) -- (m-12-4.west);
  \draw[->] (m-7-3.east) -- (m-13-4.west);
  \draw[->] (m-7-3.east) -- (m-14-4.west);
  \draw[->] (m-7-3.east) -- (m-15-4.west);
  \draw[->] (m-7-3.east) -- (m-16-4.west);
  \draw[->] (m-7-3.east) -- (m-17-4.west);
  \draw[->] (m-7-3.east) -- (m-18-4.west);

  \draw[->] (m-9-3.east) -- (m-1-4.west);
  \draw[->] (m-9-3.east) -- (m-2-4.west);
  \draw[->] (m-9-3.east) -- (m-3-4.west);
  \draw[->] (m-9-3.east) -- (m-4-4.west);
  \draw[->] (m-9-3.east) -- (m-5-4.west);
  \draw[->] (m-9-3.east) -- (m-6-4.west);
  \draw[->] (m-9-3.east) -- (m-7-4.west);
  \draw[->] (m-9-3.east) -- (m-8-4.west);
  \draw[->] (m-9-3.east) -- (m-9-4.west);
  \draw[->] (m-9-3.east) -- (m-10-4.west);
  \draw[->] (m-9-3.east) -- (m-11-4.west);
  \draw[->] (m-9-3.east) -- (m-12-4.west);
  \draw[->] (m-9-3.east) -- (m-13-4.west);
  \draw[->] (m-9-3.east) -- (m-14-4.west);
  \draw[->] (m-9-3.east) -- (m-15-4.west);
  \draw[->] (m-9-3.east) -- (m-16-4.west);
  \draw[->] (m-9-3.east) -- (m-17-4.west);
  \draw[->] (m-9-3.east) -- (m-18-4.west);
  
  \draw[->] (m-11-3.east) -- (m-12-4.west);

  \draw[->] (m-13-3.east) -- (m-15-4.west);
  
  \draw[->] (m-15-3.east) -- (m-18-4.west);

        \end{tikzpicture}
        
    \caption{Stratification of the strict equivalence orbits in \(\CC[\lambda]^{2 \times 2}\). Orbits are represented by their KCFs. An arrow from \(\MN\) to \(\MM\) denotes the relation \(\MN \in \CO(\MM)\). Colours indicate classes of pencils: pencils in class \(\dHI\) (\(\dH\) with \(Q = I\)) are shown in black, pencils in class \(\dH\) with \(Q \neq I\) in blue, and pencils outside both classes in purple. \\
    For brevity, an arbitrary eigenvalue in the open right half-plane, open left half-plane, and on the imaginary axis excluding zero is denoted by \(\lambda_r\), \(\lambda_l\), and \(\lambda_i\), respectively. If a pencil has two distinct eigenvalues belonging to the same set, the second eigenvalue is denoted by adding a hat; e.g. \(\hat{\lambda}_r\) denotes an eigenvalue in the open right half-plane satisfying \(\lambda_r \neq \hat{\lambda}_r\).
    }
    \label{fig_strat}

\end{figure}
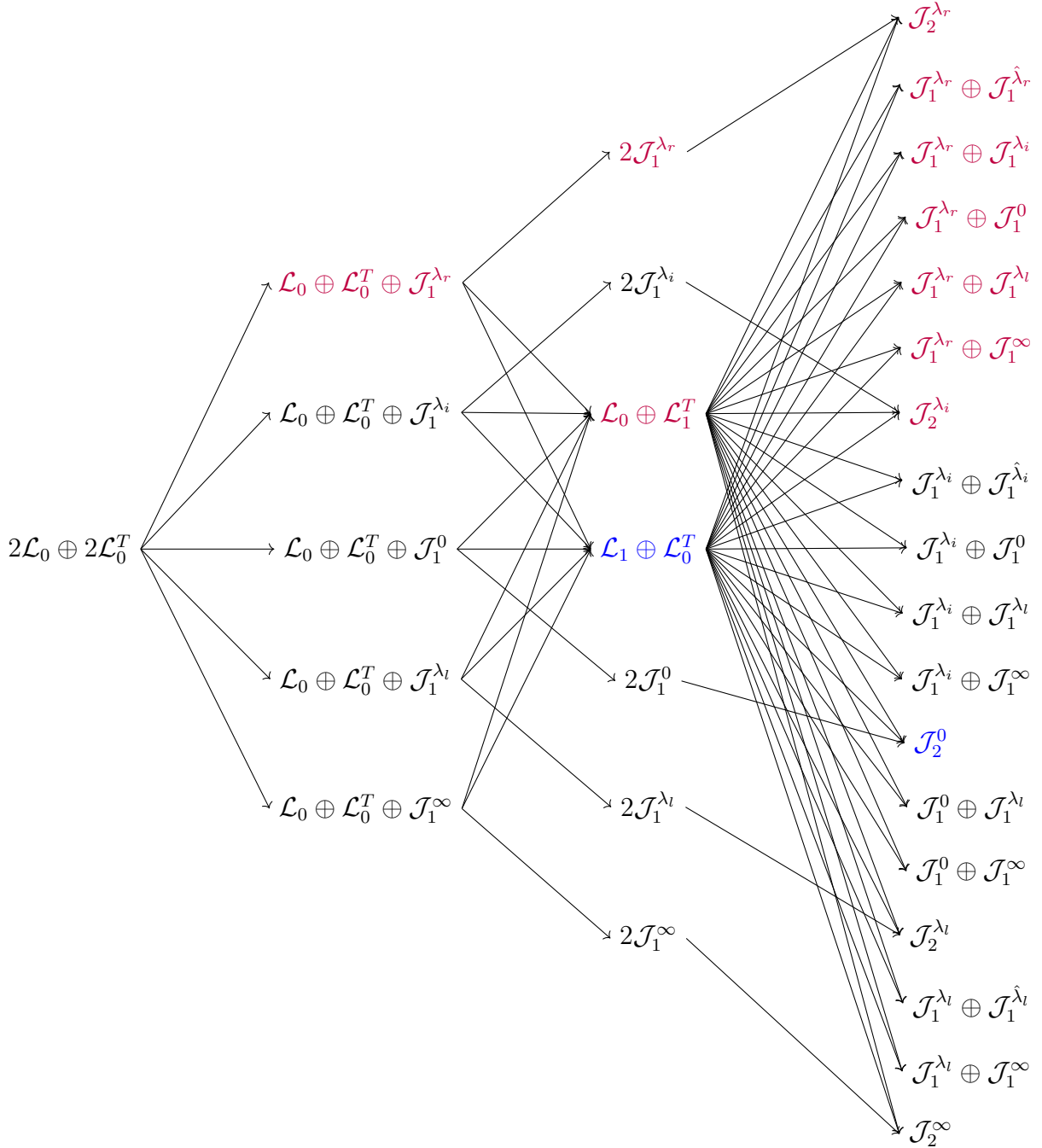

This example allows us to observe that between two orbits of class \(\dH\) (or \(\dHI\)) there may exist an orbit lying outside both classes, but not necessarily. 
To analyse the strict equivalence orbit stratification in the general case of \(n \times n\) matrix pencils, one may employ \cite[Th. II]{futorny}, which characterises the covering relations between orbits; see also \cite[Th. 2.4]{das}. This is left for future work.

The next theorem provides a characterisation, in terms of the Kronecker canonical form, of when the orbit closure of an \(n \times n\) matrix pencil belonging to class \(\dH\) (or class \(\dHI\)) does not contain any orbits lying outside of this class.

\begin{theorem} \label{th_nondegen}
\begin{enumerate}[label=\rm(\roman*)]
    \item Let \(\MM \in \CC[\lambda]^{n \times n}\) be a matrix pencil belonging to class \(\dH\). Then we have \(\CO(\MM) \subset \dH\) if and only if the KCF of pencil \(\MM\) is of the form, up to permutation of the blocks,
    \begin{equation} \label{nondegen_form}
        \underbrace{\LL_0 \oplus \ldots \oplus \LL_0}_{l \text{ blocks}}
        \oplus
        \underbrace{\LL_0^T \oplus \ldots \oplus \LL_0^T}_{l \text{ blocks}}
        \oplus
        \underbrace{\JJ_1^\mu \oplus \ldots \oplus \JJ_1^\mu}_{j \text{ blocks}},
    \end{equation}
    where \(l, j \geq 0\) are such that \(l + j = n\), and eigenvalue \(\mu\) is either satisfying \(\mathrm{Re} (\mu) \leq 0\) or is equal to \(\infty\).

    \item Let \(\MM \in \CC[\lambda]^{n \times n}\) be a matrix pencil belonging to class \(\dHI\). Then we have \(\CO(\MM) \subset \dHI\) if and only if the KCF of pencil \(\MM\) is of the form \eqref{nondegen_form}, up to permutation of the blocks.
\end{enumerate}
    
\end{theorem}

\begin{proof}
    Without loss of generality, assume that pencil \(\MM\) is in the KCF. We prove both parts of the theorem simultaneously, as the same argument applies to the classes \(\dH\) and \(\dHI\).
    
    To prove the sufficiency part, let us assume that pencil \(\MM\) is of the form \eqref{nondegen_form}. If \(j = 0\), then \(\MM\) contains only singular blocks of size \(0\) and the closure of its orbit does not contain any other orbits, as it is the minimal element of the partial order \(\prec_c\) within the set of complex \(n \times n\) matrix pencils, \(\CC[\lambda]^{n \times n}\). When \(j > 0\), observe that the only way of obtaining \(\MM\) through one of the operations \ref{R1}-\ref{R6} from Theorem \ref{th_pokrzywa} is by applying operation \ref{R6} in the following way
        \[
        \hat{\MM} \oplus \LL_0 \oplus \LL_0^T \rightsquigarrow \hat{\MM} \oplus \JJ_1^\mu = \MM,
        \]
        where \(\hat{\MM}\) is an appropriate subpencil of \(\MM\). Therefore, all pencils \(\MN\) such that \(\MN \prec_c \MM\) have KCF of the form \eqref{nondegen_form} and, by Theorem \ref{th_2classes}, belong to class \(\dHI \subset \dH\).

        On the other hand, if \(\MM\) is not of the form \eqref{nondegen_form}, we will show that there exists a pencil \(\MN\) such that \(\MN \prec_c \MM\) and \(\MN\) does not belong to class \(\dH\) (nor to subclass \(\dHI\)). First, assume that \(\MM\) contains a Jordan block of size \(k > 1\), \(\JJ_k^\nu\), where eigenvalue \(\nu\) is satisfying the conditions of Theorem \ref{th_2classes} \ref{dH_class}. Then, by operation \ref{R6} from Theorem \ref{th_pokrzywa}, for an appropriate subpencil \(\hat{\MM}\) we have
        \[
        \MN = \hat{\MM} \oplus \LL_0 \oplus \LL_{k-1}^T \prec_c \hat{\MM} \oplus \JJ_k^\nu = \MM.
        \]
        By Theorem \ref{th_2classes}, we obtain \(\MN \notin \dH\), as it contains a left singular block of size \(k - 1 \geq 1\). 
        Next, assume \(\MM\) contains Jordan blocks associated with at least two distinct eigenvalues, \(\JJ_p^{\nu_1}\), \(\JJ_q^{\nu_2}\), where \(p, q \geq 1\) and \(\nu_1 \neq \nu_2\) are in accordance with Theorem \ref{th_2classes} \ref{dH_class}. In this case, again via operation \ref{R6}, for an appropriate subpencil \(\hat{\MM}\) we obtain
        \[
        \MN = \hat{\MM} \oplus \LL_0 \oplus \LL_{p+q-1}^T \prec_c \hat{\MM} \oplus \JJ_p^{\nu_1} \oplus \JJ_q^{\nu_2} = \MM,
        \]
        where \(p + q - 1 \geq 1\). 
        Lastly, if \(\MM\) contains at least one right singular block of size one, \(\LL_1\), then as a square matrix pencil it also contains a left singular block, \(\LL_0^T\). In that case, by operation \ref{R3} we have, for an appropriate subpencil \(\hat{\MM}\),
        \[
        \MN = \hat{\MM} \oplus \LL_0 \oplus \LL_0^T \oplus \JJ_1^1 \prec_c \hat{\MM} \oplus \LL_1 \oplus \LL_0^T = \MM.
        \] 
        Here, pencil \(\MN\) has an eigenvalue lying in the open right half-plane. By Theorem \ref{th_2classes}, we obtain that \(\MN \notin \dH\).
\end{proof}

The following theorem is a version, restricted to classes \(\dH\) and \(\dHI\), of the classical result stating that the closure of an orbit is the union of the orbit itself and orbits of strictly smaller dimension.

\begin{theorem} \label{th_closure}
    Let \(\MM \in \CC[\lambda]^{n \times n}\) be a matrix pencil. Then
    \begin{enumerate}[label=\rm(\roman*),itemsep=0pt]
        \item \label{th_closure_i}
            \begin{equation} \label{closure_eq}
            \CO(\MM) \cap \dH = \bigcup_{\substack{\MN \prec_c \MM \\ \MN \in \dH}} \OO(\MN).
            \end{equation}
        \item \label{th_closure_ii}
            \begin{equation} \label{closure_eq_dHI}
            \CO(\MM) \cap \dHI = \bigcup_{\substack{\MN \prec_c \MM \\ \MN \in \dHI}} \OO(\MN).
            \end{equation}
    \end{enumerate}
\end{theorem}

\begin{proof}
    We present the proof of \ref{th_closure_i}. The proof of \ref{th_closure_ii} follows by replacing \(\dH\) by \(\dHI\). 
    
    Consider a matrix pencil \(\PP\) belonging to the union on the right-hand side of \eqref{closure_eq}, i.e. \(\PP \in \OO(\MN)\) for some pencil \(\MN\) from class \(\dH\) and such that \(\MN \prec_c \MM\). We will show that \(\PP\) belongs to the left-hand side of \eqref{closure_eq}. The converse implication is straightforward.

    Observe that \(\PP \in \dH\), as it is strictly equivalent to pencil \(\MN \in \dH\). Moreover, there exist nonsingular matrices \(V, W \in \CC^{n \times n}\) satisfying \(\PP = V \MN W \). As \(\MN \in \CO(\MM)\), let \(\{\MM_i\}_{i \in \NN} \subset \OO(\MM)\) be a sequence of matrix pencils strictly equivalent to pencil \(\MM\) and such that \(\MM_i \rightarrow \MN\) when \(i \rightarrow \infty\). Then matrix pencil \(V \MM_i W\) is strictly equivalent to \(\MM\) for all \(i \in \NN\), and with \(i \rightarrow \infty\) we have \(V \MM_i W \rightarrow V \MN W = \PP\). Therefore, \(\PP \in \CO(\MM)\), which completes the proof.
\end{proof}

Notice that if we take pencil \(\MM\) from class \(\dH\) (or class \(\dHI\)), but drop the intersection with \(\dH\) (\(\dHI\)) on the left-hand side of \eqref{closure_eq} (or \eqref{closure_eq_dHI}, respectively), the equality no longer holds. 

\begin{remark} 
    Let \(\MM \in \CC[\lambda]^{n \times n}\) be a matrix pencil belonging to class \(\dH\) \((\dHI)\). Then
    \begin{equation} \label{closure_eq_2}
    \CO(\MM) \supset \bigcup_{\substack{\MN \prec_c \MM \\ \MN \in \dH (\dHI)}} \OO(\MN).
    \end{equation}
    The fact that the other inclusion fails is a consequence of Theorem \ref{th_nondegen}.
    There is an equality in \eqref{closure_eq_2} if and only if the KCF of pencil \(\MM\) is of the form \eqref{nondegen_form}. 
\end{remark}

\begin{example}
    For a simple example of strict inclusion in \eqref{closure_eq_2} for both class \(\dH\) and class \(\dHI\), take \(\MM = \JJ_2^\infty\) and look at Figure \ref{fig_strat}. We have \(\PP \in \CO(\MM)\) for \(\PP = \LL_0 \oplus \LL_1^T\), which does not belong to class \(\dHI \subset \dH\). Therefore \(\PP\) does not belong to the right-hand side of \eqref{closure_eq_2} for neither of the two classes.
\end{example}

\section{Kronecker canonical form and properties of the pencil \(\lambda E - Q\)} \label{sec_EQ}

In this section we study pencils which are strictly equivalent to a pencil
\begin{equation} \label{eq_form_P}
    \PP = \lambda E - Q \in \CC[\lambda]^{m \times n}, \quad \text{with } \ E^*Q = Q^*E \geq 0.  
\end{equation}
Notice that here we consider possibly non-square and possibly singular matrix pencils. We extend the results presented in \cite{MMW_2018} and obtain a full characterisation of the Kronecker canonical form for pencils satisfying \eqref{eq_form_P}.

\begin{theorem} \label{th_EQ}
    A pencil \(\MM \in \CC[\lambda]^{m \times n}\) is strictly equivalent to a pencil of the form \eqref{eq_form_P}, if and only if the following conditions hold:
    \begin{enumerate}[label=\rm(\alph*), noitemsep, topsep=3pt] 
        \item The eigenvalues of \(\MM\) lie in \(\RR_{\geq0} \cup \{\infty\}\).
        \item All eigenvalues, including infinity, are semisimple.
        \item The right minimal indices are all zero (if there are any), the left minimal indices are arbitrary.
    \end{enumerate}
\end{theorem}

\begin{proof}
    For the necessity part of the proof, from \cite[Th. 3.7]{MMW_2018} we obtain that the regular part of the KCF of \(\MM\), if it exists, is strictly equivalent to a diagonal pencil of the form
        \[
        \begin{bmatrix}
        \lambda p_1 - q_1 & & &  \\
        & \lambda p_2 - q_2 & & \\
        & & \ddots & \\
        & & & \lambda p_k - q_k
        \end{bmatrix},
        \]
        with \(p_i, q_i \in \RR_{\geq 0}\) and \(p_i^2 + q_i^2 = 1\) for all \(i \in \{1, \ldots, k\}\). Recall that all unspecified entries are zeros. 
        Therefore, to establish the regular part of the KCF of pencil \(\MM\), it is sufficient to consider the \(1 \times 1\) matrix pencils \(\MM_i := \lambda p_i - q_i\). If \(p_i = 0\), we have \(|q_i| = 1 \) and pencil \(\MM_i\) is strictly equivalent to the nilpotent block \(\JJ_1^\infty = \lambda 0 - 1\). If \(p_i \neq 0\), pencil \(\MM_i\) is strictly equivalent to \(\JJ_1^{q_i / p_i} = \lambda - \frac{q_i}{p_i}\), which is a Jordan block associated with the eigenvalue \(\frac{q_i}{p_i} \in \RR_{\geq 0}\). Hence, we obtain that the regular part of the KCF of pencil \(\MM\) can contain only \(1 \times 1\) Jordan blocks associated with finite real nonnegative eigenvalues and \(1 \times 1\) nilpotent blocks associated with the eigenvalue infinity. 
        
        As for the singular part of the KCF of pencil \(\MM\), by \cite[Prop. 3.5]{MMW_2018} we obtain that the right minimal indices are all zero (if there are any), whereas the left minimal indices can be arbitrary (fitting the dimension restrictions) as a consequence of \cite[Cor. 3.9]{MMW_2018}.
        
        To prove the sufficiency part, assume without loss of generality that pencil \(\MM\) is in the KCF. First, we consider separately each type of the regular blocks which may be present in \(\MM\). A \(1 \times 1\) Jordan block associated with an eigenvalue \(\nu \in \RR_{\geq 0}\), i.e. \(\JJ_1^\nu = \lambda - \nu\), is strictly equivalent to a pencil \(\lambda E - Q\) satisfying \(E^*Q = Q^*E \geq 0\) for \(E = 1\) and \(Q = \nu\).
        In the case of \(1 \times 1\) nilpotent block associated with the eigenvalue infinity, \(\JJ_1^\infty = \lambda0 - 1\), we may take \(E = 0\) and \(Q = 1\).
        
        Regarding the singular blocks, the case of the left singular blocks follows again from \cite[Cor. 3.9]{MMW_2018}. As for the right singular blocks, assume that the number of right singular blocks of size zero in pencil \(\MM\) is equal to \(r > 0\). Since \(\MM\) is of size \(m \times n\) with \(m,n \geq 1\), it cannot consist only of the right singular blocks of size \(0\), and must contain at least one block of another type. Let us denote this block by \(\KK\), so that \(\KK\) is either equal to \(\JJ_1^\mu\) with \(\mu \in \RR_{\geq 0} \cup \{\infty \}\), or to \(\LL^T_j\) with \(j \geq 0\). Let \(\hat{\MM}\) denote the subpencil of \(\MM\) consisting of the block \(\KK\) and \(r\) blocks \(\LL_0\), that is
        \[ \hat{\MM} := 
        \begin{bmatrix}
            \ \KK & 0 & 0 & \ldots & 0 \
        \end{bmatrix}
        \]
        where there are \(r\) zero columns. We have already observed that \(\KK\) is strictly equivalent to a pencil \(\lambda E - Q\) with \(E^*Q = Q^*E \geq 0\). Now, let us consider matrices \(\hat{E}\) and \(\hat{Q}\) of the form
        \[
        \hat{E} := 
        \begin{bmatrix}
            \ E & 0 & 0 & \ldots & 0 \
        \end{bmatrix},
        \hspace{20pt}
        \hat{Q} :=
        \begin{bmatrix}
            \ Q & 0 & 0 & \ldots & 0 \
        \end{bmatrix},
        \]
        where in each the zeros denote \(r\) zero columns. To conclude the proof, notice that pencil \(\hat{\MM}\) is strictly equivalent to \(\lambda \hat{E} - \hat{Q}\), and \(\hat{E}^*\hat{Q} = \hat{Q}^*\hat{E} \geq 0\).
\end{proof}

In Figure \ref{fig_strat_2} we present the stratification of the orbits in the set \(\CC[\lambda]^{2 \times 3}\) with respect to the strict equivalence classes of pencils of the form \eqref{eq_form_P}.

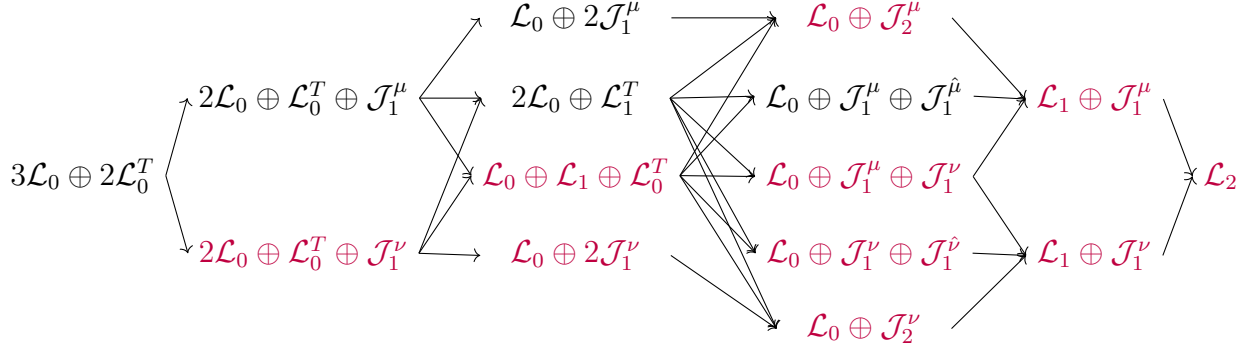
\begin{figure}[h]
    \centering
    \hspace{-0.5cm}
        \begin{tikzpicture}

    \matrix (m) [matrix of nodes, row sep=0.3cm, column sep=0.7cm, nodes in empty cells]
    {
    &[-0.4cm] & \; \(\LL_0 \oplus 2\JJ_1^\mu\) \; &[+0.3cm] \; {\color{purple} \(\LL_0 \oplus \JJ_2^\mu\)} \; & &[-0.3cm] \\
    & \(2\LL_0 \oplus \LL_0^T \oplus \JJ_1^\mu\) & \; \(2\LL_0 \oplus \LL_1^T\) \; & \(\LL_0 \oplus \JJ_1^\mu \oplus \JJ_1^{\hat{\mu}}\) & {\color{purple} \(\LL_1 \oplus \JJ_1^\mu\)} & \\
    \(3\LL_0 \oplus 2\LL_0^T \) & & {\color{purple} \(\LL_0 \oplus \LL_1 \oplus \LL_0^T\)} & {\color{purple} \(\LL_0 \oplus \JJ_1^\mu \oplus \JJ_1^\nu\)} & & {\color{purple} \(\LL_2\)} \\
    & {\color{purple} \(2\LL_0 \oplus \LL_0^T \oplus \JJ_1^\nu\)} & \; {\color{purple} \(\LL_0 \oplus 2\JJ_1^\nu\)} \; & {\color{purple} \(\LL_0 \oplus \JJ_1^\nu \oplus \JJ_1^{\hat{\nu}}\)} & {\color{purple} \(\LL_1 \oplus \JJ_1^\nu\)} & \\
    & & & \; {\color{purple} \(\LL_0 \oplus \JJ_2^\nu\)} \; & & \\
    };

    \draw[->] (m-3-1.east) -- (m-2-2.west);
    \draw[->] (m-3-1.east) -- (m-4-2.west);
  
    \draw[->] (m-2-2.east) -- (m-1-3.west);
    \draw[->] (m-2-2.east) -- (m-2-3.west);
    \draw[->] (m-2-2.east) -- (m-3-3.west);
  
    \draw[->] (m-4-2.east) -- (m-2-3.west);
    \draw[->] (m-4-2.east) -- (m-3-3.west);
    \draw[->] (m-4-2.east) -- (m-4-3.west);
  
    \draw[->] (m-1-3.east) -- (m-1-4.west);

    \draw[->] (m-2-3.east) -- (m-1-4.west);
    \draw[->] (m-2-3.east) -- (m-2-4.west);
    \draw[->] (m-2-3.east) -- (m-3-4.west);
    \draw[->] (m-2-3.east) -- (m-4-4.west);
    \draw[->] (m-2-3.east) -- (m-5-4.west);

    \draw[->] (m-3-3.east) -- (m-1-4.west);
    \draw[->] (m-3-3.east) -- (m-2-4.west);
    \draw[->] (m-3-3.east) -- (m-3-4.west);
    \draw[->] (m-3-3.east) -- (m-4-4.west);
    \draw[->] (m-3-3.east) -- (m-5-4.west);

    \draw[->] (m-4-3.east) -- (m-5-4.west);

    \draw[->] (m-1-4.east) -- (m-2-5.west);
    
    \draw[->] (m-2-4.east) -- (m-2-5.west);

    \draw[->] (m-3-4.east) -- (m-2-5.west);
    \draw[->] (m-3-4.east) -- (m-4-5.west);
    
    \draw[->] (m-4-4.east) -- (m-4-5.west);
    
    \draw[->] (m-5-4.east) -- (m-4-5.west);

    \draw[->] (m-2-5.east) -- (m-3-6.west);

    \draw[->] (m-4-5.east) -- (m-3-6.west);    

        \end{tikzpicture}

    \caption{Stratification of the strict equivalence orbits in \(\CC[\lambda]^{2 \times 3}\). Orbits are represented by their KCFs. An arrow from \(\KK\) to \(\MM\) denotes the relation \(\KK \in \CO(\MM)\). Black denotes orbits of pencils of the form \eqref{eq_form_P}.
    Arbitrary distinct eigenvalues in \(\RR_{\geq 0} \cup \{\infty\}\) are denoted by \(\mu\) and \(\hat{\mu}\), whereas \(\nu\) and \(\hat{\nu}\) denote arbitrary distinct eigenvalues in \(\CC \setminus \RR_{\geq 0}\).}
    \label{fig_strat_2}
\end{figure}

Notice that the maximal element of the stratification in Figure \ref{fig_strat_2}, with the KCF consisting of a single block \(\LL_2\), is not strictly equivalent to a pencil of the form \eqref{eq_form_P} by Theorem \ref{th_EQ}. However, in some cases we can obtain that property, as for example if we considered the set \(\CC[\lambda]^{3 \times 2}\), the maximal orbit would correspond to the KCF \(\LL_2^T\).  

In the next theorem we characterise, in terms of the Kronecker canonical form, when the orbit closure of a pencil \(\PP\) as in \eqref{eq_form_P} contains only orbits of pencils which also have the form \eqref{eq_form_P}. It can be proved analogously to Theorem \ref{th_nondegen}.

\begin{theorem}
    Let \(\MM \in \CC[\lambda]^{m \times n}\) be strictly equivalent to a pencil \(\PP\) as in \eqref{eq_form_P}. Then for every \(\MN \prec_c \MM\) there exists a pencil \(\hat{\PP}\) of the form \eqref{eq_form_P} strictly equivalent to \(\MN\), if and only if the KCF of pencil \(\MM\) is of the form, up to permutation of the blocks,
    \[
    \underbrace{\LL_0 \oplus \ldots \oplus \LL_0}_{k \text{ blocks}}
    \oplus
    \underbrace{\LL_0^T \oplus \ldots \oplus \LL_0^T}_{l \text{ blocks}}
    \oplus
    \underbrace{\JJ_1^\mu \oplus \ldots \oplus \JJ_1^\mu}_{j \text{ blocks}},
    \]    
    where \(k, l, j \geq 0\) are such that \(k + j = n\), \(l + j = m\), and \(\mu \in \RR_{\geq 0} \cup \{\infty\}\).
    
\end{theorem}

\section{Properties of the pencil \(\lambda E - (J - R)Q\) with possibly singular \(\lambda E - Q\)} \label{sec_singular}

In this section we consider square matrix pencils \(\MM = \lambda E - (J - R)Q\) of the form \eqref{eq_form_general} without any additional assumptions about the regularity of \(\lambda E - Q\). We obtain that the KCF of such \(\MM\) may contain all types and sizes of the canonical blocks, although it cannot be arbitrary.

\begin{remark}
    Notice that the form \eqref{eq_form_general} is preserved under strict equivalence. Namely, consider a pencil \(\MM = \lambda E - (J - R)Q\) satisfying \eqref{eq_form_general}, and let \(\PP\) be a pencil strictly equivalent to \(\MM\). Then there exist nonsingular matrices \(V\) and \(W\) such that \(\PP = V \MM W\). Defining
    \[
        \hat{E} := VEW, \quad \hat{J} := VJV^*, \quad \hat{R} := VRV^*, \quad \text{and } \hat{Q} := (V^*)^{-1}QW,
    \]
    we obtain \(\PP = \lambda \hat{E} - (\hat{J} - \hat{R})\hat{Q}\), which also satisfies \eqref{eq_form_general}. Consequently, all pencils in class \(\dH\) are of the form \eqref{eq_form_general}. See e.g. \cite{blazhko,gillis} for related results.

\end{remark}

In Lemmas \ref{lem_right}, \ref{lem_jordan}, and \ref{lem_nilpotent}, we explicitly construct pencils of the form \eqref{eq_form_general} that have an arbitrary right minimal index, or contain in the KCF an arbitrary Jordan or nilpotent block, respectively. The construction of a pencil with an arbitrary left index is already available in the literature \cite[Ex. 4.8]{MMW_2018}. As a consequence, we obtain the following result.

\begin{theorem} \label{th_singular}
    Let \(\KK\) be one of the following:
    \begin{enumerate}[label=\rm(\alph*), noitemsep, topsep=3pt]
        \item \label{case_right} a right singular block \(\LL_k\) of size \(k \geq 0\);
        \item \label{case_left} a left singular block \(\LL_k^T\) of size \(k \geq 0\);
        \item \label{case_jordan} a Jordan block \(\JJ_k^\mu\) of size \(k \geq 1\) associated with a finite eigenvalue \(\mu \in \CC\);
        \item \label{case_nilpotent} a nilpotent block \(\JJ_k^\infty\) of size \(k \geq 1\) associated with the eigenvalue infinity.
    \end{enumerate}
    Then, for sufficiently large \(n\), there exists a pencil \(\MM \in \CC[\lambda]^{n \times n}\) satisfying \eqref{eq_form_general}, such that the KCF of \(\MM\) contains the block \(\KK\).
\end{theorem}

    Observe that the minimal \(n\) for which such a pencil \(\MM\) exists is equal to \(k\) when \(\KK\) is a square block appearing in the KCF of dH pencils, and equal to \(k + 1\) when \(\KK\) is a right singular block of size \(k \leq 1\) or an arbitrary left singular block; see Theorem \ref{th_2classes} and \cite[Ex. 4.8]{MMW_2018}. In the remaining cases, one has \(n \leq 2k\); see the following lemmas.

\begin{lemma} \label{lem_right}
    Let \(k \geq 1\) and consider the \(2k \times 2k\) matrices
    \begin{equation} \label{eq_JR}
    J = \left[
        \begin{NiceMatrix}
            & & & & & 1 \\
            & 0 & & & \iddots & \\
            & & & 1 & & \\ 
            & & -1 & & & \\
            & \raisebox{0.8ex}{\(\iddots\)} & & & 0 & \\
            -1 & & & & & \\
        \CodeAfter
        \tikz \draw[dash pattern=on 2pt off 2pt] (4-|1) -- (4-|last);
        \tikz \draw[dash pattern=on 2pt off 2pt] (1-|4) -- (last-|4);
        \tikz \draw[decorate, decoration={brace, amplitude=5pt}]
        ([yshift=-3pt]last-|3.9) -- node[below=4pt] {\(k\)} ([yshift=-3pt]last-|1);
        \tikz \draw[decorate, decoration={brace, amplitude=5pt}]
        ([yshift=-3pt]last-|last) -- node[below=4pt] {\(k\)} ([yshift=-3pt]last-|4.1);
        \tikz \draw[decorate, decoration={brace,amplitude=5pt}]
        ([xshift=7pt]1.1-|last) -- node[right=4pt] {\(k\)} ([xshift=7pt]3.9-|last);
        \tikz \draw[decorate, decoration={brace,amplitude=5pt}]
        ([xshift=7pt]4.1-|last) -- node[right=4pt] {\(k\)} ([xshift=7pt]7.9-|last);
        \end{NiceMatrix}
    \right],
    \vspace{20pt}
    \quad \quad
    R = 0,
    \vspace{5pt}        
    \end{equation}
    
    \[
    E = \left[
        \begin{NiceMatrix}
            1 & & & \ & & \ \ \\
            & \ddots & & & 0 & \\
            & & 1 & & & \\ 
            & & & & & \\
            & 0 & & & 0 & \\
            & & & & & 
        \CodeAfter
        \tikz \draw[dash pattern=on 2pt off 2pt]
        (4-|1) -- (4-|last);
        \tikz \draw[dash pattern=on 2pt off 2pt]
        (1-|4) -- (last-|4);
        \tikz \draw[decorate, decoration={brace, amplitude=5pt}]
        ([yshift=-3pt]last-|3.9) -- node[below=4pt] {\(k\)} ([yshift=-3pt]last-|1);
        \tikz \draw[decorate, decoration={brace, amplitude=5pt}]
        ([yshift=-3pt]last-|last) -- node[below=4pt] {\(k\)} ([yshift=-3pt]last-|4.1);
        \tikz \draw[decorate, decoration={brace,amplitude=5pt}]
        ([xshift=7pt]1.1-|last) -- node[right=4pt] {\(k\)} ([xshift=7pt]3.9-|last);
        \tikz \draw[decorate, decoration={brace,amplitude=5pt}]
        ([xshift=7pt]4.1-|last) -- node[right=4pt] {\(k\)} ([xshift=7pt]7.9-|last);
        \end{NiceMatrix}
        \right],
    \quad \text{ and } \
    Q = \left[
        \begin{NiceMatrix}
            & & & & \ & & \ \ \\
            & \ \ \ 0 & & & & 0 & \\
            & & & & & & \\ 
            & & 0 & 1 & & & \\
            & \iddots & \iddots & & & 0 & \\
            0 & 1 & & & & & 
        \CodeAfter
        \tikz \draw[dash pattern=on 2pt off 2pt] (4-|1) -- (4-|last);
        \tikz \draw[dash pattern=on 2pt off 2pt] (1-|5) -- (last-|5);
        \tikz \draw[decorate, decoration={brace, amplitude=5pt}]
        ([yshift=-3pt]last-|4.9) -- node[below=4pt] {\(k+1\)} ([yshift=-3pt]last-|1);
        \tikz \draw[decorate, decoration={brace, amplitude=5pt}]
        ([yshift=-3pt]last-|last) -- node[below=4pt] {\(k-1\)} ([yshift=-3pt]last-|5.1);
        \tikz \draw[decorate, decoration={brace,amplitude=5pt}]
        ([xshift=7pt]1.1-|last) -- node[right=4pt] {\(k\)} ([xshift=7pt]3.9-|last);
        \tikz \draw[decorate, decoration={brace,amplitude=5pt}]
        ([xshift=7pt]4.1-|last) -- node[right=4pt] {\(k\)} ([xshift=7pt]7.9-|last);
        \end{NiceMatrix}
    \right].
    \vspace{30pt}
    \]
    Then the pencil \(\lambda E - Q\) is strictly equivalent to \(\JJ_1^0 \oplus \JJ_1^\infty \oplus (k-1) \LL_0 \oplus (k-1) \LL_1^T\), with \(E^*Q = Q^*E \geq 0\). Moreover,
    \[
    \lambda E - (J-R)Q = \LL_k \oplus (k-1) \LL_0 \oplus k \LL_0^T.
    \]
\end{lemma}

\begin{proof}
    Observe that \(E^*Q = Q^*E = 0\). Moreover, one can easily verify that the pencil \(\lambda E - (J -R)Q\) satisfies the described form. We will show that the pencil \(\lambda E - Q\) is of the indicated KCF. We obtain
        \[
        \lambda E - Q = \left[
        \begin{NiceMatrix}
            \lambda & & & & & \ & & \ \ \\
            & \lambda  & & & 0 & & 0 & \\
            & & \ddots  & & & & & \\
            & & & \lambda & & & & \\ 
            & & & 0 & -1 & & & \\
            & & 0 & -1 & & & & \\
            & \iddots & \iddots & & & & 0 & \\
            0 & -1 & & & & & &
        \CodeAfter
        \tikz \draw[dash pattern=on 2pt off 2pt]
        (5-|1) -- (5-|last);
        \tikz \draw[dash pattern=on 2pt off 2pt]
        (1-|5) -- (last-|5);
        \tikz \draw[dash pattern=on 2pt off 2pt]
        (1-|6) -- (last-|6);
        \tikz \draw[decorate, decoration={brace, amplitude=5pt}]
        ([yshift=-3pt]last-|4.9) -- node[below=4pt] {\(k\)} ([yshift=-3pt]last-|1);
        \tikz \draw[decorate, decoration={brace, amplitude=5pt}]
        ([yshift=-3pt]last-|6) -- node[below=4pt] {\(1\)} ([yshift=-3pt]last-|5);
        \tikz \draw[decorate, decoration={brace, amplitude=5pt}]
        ([yshift=-3pt]last-|last) -- node[below=4pt] {\(k-1\)} ([yshift=-3pt]last-|6.2);
        \tikz \draw[decorate, decoration={brace,amplitude=5pt}]
        ([xshift=7pt]1.1-|last) -- node[right=4pt] {\(k\)} ([xshift=7pt]4.9-|last);
        \tikz \draw[decorate, decoration={brace,amplitude=5pt}]
        ([xshift=7pt]5.1-|last) -- node[right=4pt] {\(k\)} ([xshift=7pt]8.9-|last);
        \end{NiceMatrix}
        \right],
        \vspace{20pt}
        \]
        which, by a permutation of the rows, is strictly equivalent to the pencil
        \[
        \left[
        \begin{NiceMatrix}
            \rule{0pt}{3ex} \ \lambda \ & & 0 & & 0 & \ & 0 & \ \ \\
            & \lambda & & & & & & \\
            & -1 & & & & & & \\ 
            0 & & \rule[-3ex]{0pt}{8ex} \ddots & & 0 & & 0 & \\
            & & & \lambda & & & & \\
            & & & -1 & & & & \\
            \rule{0pt}{3ex} 0 & & 0 & & -1 & & 0 & \\
        \CodeAfter
        \tikz \draw[dash pattern=on 2pt off 2pt]
        (2-|1) -- (2-|last);
        \tikz \draw[dash pattern=on 2pt off 2pt]
        (7-|1) -- (7-|last);
        \tikz \draw[dash pattern=on 2pt off 2pt]
        (1-|2) -- (last-|2);
        \tikz \draw[dash pattern=on 2pt off 2pt]
        (1-|5) -- (last-|5);
        \tikz \draw[dash pattern=on 2pt off 2pt]
        (1-|6) -- (last-|6);
        \tikz \draw[decorate, decoration={brace, amplitude=5pt}]
        ([yshift=-3pt]last-|2) -- node[below=4pt] {\(1\)} ([yshift=-3pt]last-|1);
        \tikz \draw[decorate, decoration={brace, amplitude=5pt}]
        ([yshift=-3pt]last-|4.9) -- node[below=4pt] {\(k-1\)} ([yshift=-3pt]last-|2.1);
        \tikz \draw[decorate, decoration={brace, amplitude=5pt}]
        ([yshift=-3pt]last-|6) -- node[below=4pt] {\(1\)} ([yshift=-3pt]last-|5);
        \tikz \draw[decorate, decoration={brace, amplitude=5pt}]
        ([yshift=-3pt]last-|last) -- node[below=4pt] {\(k-1\)} ([yshift=-3pt]last-|6.2);
        \tikz \draw[decorate, decoration={brace,amplitude=5pt}]
        ([xshift=-7pt]2-|1) -- node[left=4pt] {\(1\)} ([xshift=-7pt]1.1-|1);
        \tikz \draw[decorate, decoration={brace,amplitude=5pt}]
        ([xshift=-7pt]6.8-|1) -- node[left=4pt] {\(2k - 2\)} ([xshift=-7pt]2.2-|1);
        \tikz \draw[decorate, decoration={brace,amplitude=5pt}]
        ([xshift=-7pt]7.9-|1) -- node[left=4pt] {\(1\)} ([xshift=-7pt]7-|1);
        \end{NiceMatrix}
        \right]
        = \JJ_1^0 \oplus (k-1) \LL_1^T \oplus \JJ_1^\infty \oplus (k-1) \LL_0.
        \vspace{20pt}
        \]
\end{proof}

\begin{lemma} \label{lem_jordan}
    Let \(k \geq 1\), \(\mu \in \CC\), and consider the \(2k \times 2k\) matrices \(J, R\) as in \eqref{eq_JR},
    \[
    E = \left[
        \begin{NiceMatrix}
            1 & & & \ & & \ \ \\
            & \ddots & & & 0 & \\
            & & 1 & & & \\ 
            & & & & & \\
            & 0 & & & 0 & \\
            & & & & & 
        \CodeAfter
        \tikz \draw[dash pattern=on 2pt off 2pt]
        (4-|1) -- (4-|last);
        \tikz \draw[dash pattern=on 2pt off 2pt]
        (1-|4) -- (last-|4);
        \tikz \draw[decorate, decoration={brace, amplitude=5pt}]
        ([yshift=-3pt]last-|3.9) -- node[below=4pt] {\(k\)} ([yshift=-3pt]last-|1);
        \tikz \draw[decorate, decoration={brace, amplitude=5pt}]
        ([yshift=-3pt]last-|last) -- node[below=4pt] {\(k\)} ([yshift=-3pt]last-|4.1);
        \tikz \draw[decorate, decoration={brace,amplitude=5pt}]
        ([xshift=7pt]1.1-|last) -- node[right=4pt] {\(k\)} ([xshift=7pt]3.9-|last);
        \tikz \draw[decorate, decoration={brace,amplitude=5pt}]
        ([xshift=7pt]4.1-|last) -- node[right=4pt] {\(k\)} ([xshift=7pt]7.9-|last);
        \end{NiceMatrix}
        \right],
    \quad \text{ and } \
    Q = \left[
        \begin{NiceMatrix}
            & & & & \ & & \ \ \\
            & \ \ \ 0 & & & & 0 & \\
            & & & & & & \\ 
            & & & \mu & & & \\
            & & \mu & 1 & & & \\
            & \iddots & \iddots & & & 0 & \\
            \mu & 1 & & & & & 
        \CodeAfter
        \tikz \draw[dash pattern=on 2pt off 2pt] (4-|1) -- (4-|last);
        \tikz \draw[dash pattern=on 2pt off 2pt] (1-|5) -- (last-|5);
        \tikz \draw[decorate, decoration={brace, amplitude=5pt}]
        ([yshift=-3pt]last-|4.9) -- node[below=4pt] {\(k\)} ([yshift=-3pt]last-|1);
        \tikz \draw[decorate, decoration={brace, amplitude=5pt}]
        ([yshift=-3pt]last-|last) -- node[below=4pt] {\(k\)} ([yshift=-3pt]last-|5.1);
        \tikz \draw[decorate, decoration={brace,amplitude=5pt}]
        ([xshift=7pt]1.1-|last) -- node[right=4pt] {\(k\)} ([xshift=7pt]3.9-|last);
        \tikz \draw[decorate, decoration={brace,amplitude=5pt}]
        ([xshift=7pt]4.1-|last) -- node[right=4pt] {\(k\)} ([xshift=7pt]7.9-|last);
        \end{NiceMatrix}
    \right].
    \vspace{20pt}
    \]
    Then the pencil \(\lambda E - Q\) is strictly equivalent to \(\JJ_1^0 \oplus k \LL_0 \oplus \LL_0^T \oplus (k-1) \LL_1^T\) if \(\mu = 0\), and to \(k \LL_0 \oplus k \LL_1^T\) if \(\mu \neq 0\). Moreover, \(E^*Q = Q^*E \geq 0\), and
    \[
    \lambda E - (J-R)Q = \JJ_k^\mu \oplus k \LL_0 \oplus k \LL_0^T.
    \]
\end{lemma}

\begin{proof}
    We have \(E^*Q = Q^*E = 0\) and it is straightforward that the pencil \(\lambda E - (J -R)Q\) is of the described form. To verify that the pencil \(\lambda E - Q\) is strictly equivalent to the indicated KCF, first observe that
        \[
        \lambda E - Q = \left[
        \begin{NiceMatrix}
            \lambda & & & & \ & & \ \ \\
            & \lambda & & & & 0 & \\
            & & \ddots & & & & \\
            & & & \lambda & & & \\ 
            & & & -\mu & & & \\
            & & -\mu & -1 & & & \\
            & \iddots & \iddots & & & 0 & \\
            -\mu & -1 & & & & &
        \CodeAfter
        \tikz \draw[dash pattern=on 2pt off 2pt]
        (5-|1) -- (5-|last);
        \tikz \draw[dash pattern=on 2pt off 2pt]
        (1-|5) -- (last-|5);
        \tikz \draw[decorate, decoration={brace, amplitude=5pt}]
        ([yshift=-3pt]last-|4.9) -- node[below=4pt] {\(k\)} ([yshift=-3pt]last-|1);
        \tikz \draw[decorate, decoration={brace, amplitude=5pt}]
        ([yshift=-3pt]last-|last) -- node[below=4pt] {\(k\)} ([yshift=-3pt]last-|5);
        \tikz \draw[decorate, decoration={brace,amplitude=5pt}]
        ([xshift=7pt]1.1-|last) -- node[right=4pt] {\(k\)} ([xshift=7pt]4.9-|last);
        \tikz \draw[decorate, decoration={brace,amplitude=5pt}]
        ([xshift=7pt]5.1-|last) -- node[right=4pt] {\(k\)} ([xshift=7pt]8.9-|last);
        \end{NiceMatrix}
        \right].
        \vspace{20pt}
        \]
        When \(\mu = 0\), by a permutation of the rows, the pencil \(\lambda E - Q\) is strictly equivalent to
        \[
        \left[
        \begin{NiceMatrix}
            \rule{0pt}{3ex} \ \lambda \ & & 0 & & \ & 0 & \ \ \\
            & \lambda & & & & & \\
            & -1 & & & & & \\ 
            0 & & \rule[-3ex]{0pt}{8ex} \ddots & & & 0 & \\
            & & & \lambda & & & \\
            & & & -1 & & & \\
            \rule{0pt}{3ex} 0 & & 0 & & & 0 & \\
        \CodeAfter
        \tikz \draw[dash pattern=on 2pt off 2pt]
        (2-|1) -- (2-|last);
        \tikz \draw[dash pattern=on 2pt off 2pt]
        (7-|1) -- (7-|last);
        \tikz \draw[dash pattern=on 2pt off 2pt]
        (1-|2) -- (last-|2);
        \tikz \draw[dash pattern=on 2pt off 2pt]
        (1-|5) -- (last-|5);
        \tikz \draw[decorate, decoration={brace, amplitude=5pt}]
        ([yshift=-3pt]last-|2) -- node[below=4pt] {\(1\)} ([yshift=-3pt]last-|1);
        \tikz \draw[decorate, decoration={brace, amplitude=5pt}]
        ([yshift=-3pt]last-|4.9) -- node[below=4pt] {\(k-1\)} ([yshift=-3pt]last-|2.1);
        \tikz \draw[decorate, decoration={brace, amplitude=5pt}]
        ([yshift=-3pt]last-|last) -- node[below=4pt] {\(k\)} ([yshift=-3pt]last-|5.1);
        \tikz \draw[decorate, decoration={brace,amplitude=5pt}]
        ([xshift=-7pt]2-|1) -- node[left=4pt] {\(1\)} ([xshift=-7pt]1.1-|1);
        \tikz \draw[decorate, decoration={brace,amplitude=5pt}]
        ([xshift=-7pt]6.8-|1) -- node[left=4pt] {\(2(k - 1)\)} ([xshift=-7pt]2.2-|1);
        \tikz \draw[decorate, decoration={brace,amplitude=5pt}]
        ([xshift=-7pt]7.9-|1) -- node[left=4pt] {\(1\)} ([xshift=-7pt]7-|1);
        \end{NiceMatrix}
        \right]
        = \JJ_1^0 \oplus (k-1) \LL_1^T \oplus \LL_0^T \oplus k \LL_0.
        \vspace{20pt}
        \]
        If \(\mu \neq 0\), we reduce the minus ones in the lower half of the pencil \(\lambda E - Q\) and then divide the rows \(k+1, \ldots, 2k\) by \(\mu\). Now, again by permuting the rows, we obtain that the pencil \(\lambda E - Q\) is strictly equivalent to
        \[
        \left[
        \begin{NiceMatrix}
            \lambda & & & \ & & \ \ \\
            -1 & & & & & \\ 
            & \rule[-3ex]{0pt}{8ex} \ddots & & & 0 & \\
            & & \lambda & & & \\
            & & -1 & & & \\
        \CodeAfter
        \tikz \draw[dash pattern=on 2pt off 2pt]
        (1-|4) -- (last-|4);
        \tikz \draw[decorate, decoration={brace, amplitude=5pt}]
        ([yshift=-3pt]last-|3.9) -- node[below=4pt] {\(k\)} ([yshift=-3pt]last-|1);
        \tikz \draw[decorate, decoration={brace, amplitude=5pt}]
        ([yshift=-3pt]last-|last) -- node[below=4pt] {\(k\)} ([yshift=-3pt]last-|4.1);
        \tikz \draw[decorate, decoration={brace,amplitude=5pt}]
        ([xshift=-7pt]5.9-|1) -- node[left=4pt] {\(2k\)} ([xshift=-7pt]1.1-|1);
        \end{NiceMatrix}
        \right]
        = k \LL_1^T \oplus k \LL_0.
        \vspace{10pt}
        \]
\end{proof}

\begin{lemma} \label{lem_nilpotent}
    Let \(k \geq 1\) and consider the \(2k \times 2k\) matrices \(J, R\) as in \eqref{eq_JR},
    \[
    E = \left[
        \begin{NiceMatrix}
            0 & 1 & & & \ & & \ \ \\
            & 0 & \raisebox{-0.8ex}{\(\ddots\)} & & & 0 & \\
            & & \ddots & 1 & & & \\ 
            & & & 0 & & & \\
            & & & & & & \\
            & \ \ 0 & & & & 0 & \\
            & & & & & & 
        \CodeAfter
        \tikz \draw[dash pattern=on 2pt off 2pt]
        (5-|1) -- (5-|last);
        \tikz \draw[dash pattern=on 2pt off 2pt]
        (1-|5) -- (last-|5);
        \tikz \draw[decorate, decoration={brace, amplitude=5pt}]
        ([yshift=-3pt]last-|4.9) -- node[below=4pt] {\(k\)} ([yshift=-3pt]last-|1);
        \tikz \draw[decorate, decoration={brace, amplitude=5pt}]
        ([yshift=-3pt]last-|last) -- node[below=4pt] {\(k\)} ([yshift=-3pt]last-|5.1);
        \tikz \draw[decorate, decoration={brace,amplitude=5pt}]
        ([xshift=7pt]1.1-|last) -- node[right=4pt] {\(k\)} ([xshift=7pt]4.9-|last);
        \tikz \draw[decorate, decoration={brace,amplitude=5pt}]
        ([xshift=7pt]5.1-|last) -- node[right=4pt] {\(k\)} ([xshift=7pt]7.9-|last);
        \end{NiceMatrix}
        \right],
        \vspace{20pt}
    \quad \text{ and } \
    Q = \left[
        \begin{NiceMatrix}
            & & & \ & & \ \ \\
            & 0 & & & 0 & \\ 
            & & & & & \\
            & & 1 & & & \\
            & \iddots & & & 0 & \\
            1 & & & & &
        \CodeAfter
        \tikz \draw[dash pattern=on 2pt off 2pt] (4-|1) -- (4-|last);
        \tikz \draw[dash pattern=on 2pt off 2pt] (1-|4) -- (last-|4);
        \tikz \draw[decorate, decoration={brace, amplitude=5pt}]
        ([yshift=-3pt]last-|3.9) -- node[below=4pt] {\(k\)} ([yshift=-3pt]last-|1);
        \tikz \draw[decorate, decoration={brace, amplitude=5pt}]
        ([yshift=-3pt]last-|last) -- node[below=4pt] {\(k\)} ([yshift=-3pt]last-|4.1);
        \tikz \draw[decorate, decoration={brace,amplitude=5pt}]
        ([xshift=7pt]1.1-|last) -- node[right=4pt] {\(k\)} ([xshift=7pt]3.9-|last);
        \tikz \draw[decorate, decoration={brace,amplitude=5pt}]
        ([xshift=7pt]4.1-|last) -- node[right=4pt] {\(k\)} ([xshift=7pt]7.9-|last);
        \end{NiceMatrix}
    \right],
    \]
    Then the pencil \(\lambda E - Q\) is strictly equivalent to \(\JJ_1^\infty \oplus k \LL_0 \oplus \LL_0^T \oplus (k-1) \LL_1^T\), with \(E^*Q = Q^*E \geq 0\). Moreover,
    \[
    \lambda E - (J-R)Q = \JJ_k^\infty \oplus k \LL_0 \oplus k \LL_0^T.
    \]
\end{lemma}

\begin{proof}
     We obtain that \(E^*Q = Q^*E = 0\). Similarly to previous lemmas, it is easy to verify that the pencil \(\lambda E - (J -R)Q\) satisfies the described form. Hence, it remains to show that the pencil \(\lambda E - Q\) is of the indicated KCF.
     We have
        \[
        \hspace*{-3cm}
        \lambda E - Q = \left[
        \begin{NiceMatrix}
            0 & \lambda & & & \ & & \ \ \\
            & 0 & \raisebox{-0.8ex}{\(\ddots\)} & & & 0 & \\
            & & \ddots & \lambda & & & \\ 
            & & & 0 & & & \\
            & & & -1 & & & \\
            & & \iddots & & & & \\
            & -1 & & & & 0 & \\
            -1 & & & & & & 
        \CodeAfter
        \tikz \draw[dash pattern=on 2pt off 2pt]
        (5-|1) -- (5-|last);
        \tikz \draw[dash pattern=on 2pt off 2pt]
        (1-|5) -- (last-|5);
        \tikz \draw[decorate, decoration={brace, amplitude=5pt}]
        ([yshift=-3pt]last-|4.9) -- node[below=4pt] {\(k\)} ([yshift=-3pt]last-|1);
        \tikz \draw[decorate, decoration={brace, amplitude=5pt}]
        ([yshift=-3pt]last-|last) -- node[below=4pt] {\(k\)} ([yshift=-3pt]last-|5);
        \tikz \draw[decorate, decoration={brace,amplitude=5pt}]
        ([xshift=7pt]1.2-|last) -- node[right=4pt] {\(k\)} ([xshift=7pt]4.9-|last);
        \tikz \draw[decorate, decoration={brace,amplitude=5pt}]
        ([xshift=7pt]5.1-|last) -- node[right=4pt] {\(k\)} ([xshift=7pt]8.8-|last);
        \end{NiceMatrix}
        \right].
        \vspace{25pt}
        \]
        By permuting the rows, \(\lambda E - Q\) is strictly equivalent to the pencil
        \[
        \left[
        \begin{NiceMatrix}
            \rule{0pt}{3ex} -1 & & 0 & & \ & 0 & \ \ \\
            & \lambda & & & & & \\
            & -1 & & & & & \\ 
            0 & & \rule[-3ex]{0pt}{8ex} \ddots & & & 0 & \\
            & & & \lambda & & & \\
            & & & -1 & & & \\
            \rule{0pt}{3ex} 0 & & 0 & & &  0 & \\
        \CodeAfter
        \tikz \draw[dash pattern=on 2pt off 2pt]
        (2-|1) -- (2-|last);
        \tikz \draw[dash pattern=on 2pt off 2pt]
        (7-|1) -- (7-|last);
        \tikz \draw[dash pattern=on 2pt off 2pt]
        (1-|2) -- (last-|2);
        \tikz \draw[dash pattern=on 2pt off 2pt]
        (1-|5) -- (last-|5);
        \tikz \draw[decorate, decoration={brace, amplitude=5pt}]
        ([yshift=-3pt]last-|2) -- node[below=4pt] {\(1\)} ([yshift=-3pt]last-|1);
        \tikz \draw[decorate, decoration={brace, amplitude=5pt}]
        ([yshift=-3pt]last-|4.9) -- node[below=4pt] {\(k-1\)} ([yshift=-3pt]last-|2.1);
        \tikz \draw[decorate, decoration={brace, amplitude=5pt}]
        ([yshift=-3pt]last-|last) -- node[below=4pt] {\(k\)} ([yshift=-3pt]last-|5.1);
        \tikz \draw[decorate, decoration={brace,amplitude=5pt}]
        ([xshift=-7pt]2-|1) -- node[left=4pt] {\(1\)} ([xshift=-7pt]1.1-|1);
        \tikz \draw[decorate, decoration={brace,amplitude=5pt}]
        ([xshift=-7pt]6.8-|1) -- node[left=4pt] {\(2(k - 1)\)} ([xshift=-7pt]2.2-|1);
        \tikz \draw[decorate, decoration={brace,amplitude=5pt}]
        ([xshift=-7pt]7.9-|1) -- node[left=4pt] {\(1\)} ([xshift=-7pt]7-|1);
        \end{NiceMatrix}
        \right]
        = \JJ_1^{\infty} \oplus (k-1) \LL_1^T \oplus \LL_0^T \oplus k \LL_0.
        \vspace{10pt}
        \]
\end{proof}

As already mentioned, even if \(\lambda E - Q\) is allowed to be singular, pencil \(\MM = \lambda E - (J - R) Q\) satisfying \eqref{eq_form_general} cannot have an arbitrary KCF. First of all, by \cite[Prop. 4.1]{MMW_2018}, we obtain that if \(\lambda E - Q\) is singular, then also pencil \(\MM\) is singular. This might suggest that pencil \(\MM\) could have an arbitrary singular KCF, but even that restriction is not sufficient, as the next example shows.

\begin{example}
    Let \(\MN =  \LL_2 \oplus \LL_0^T\). Assume \(\MN\) is the KCF of a pencil \(\MM = \lambda E - (J - R)Q\) satisfying \eqref{eq_form_general}. By Theorem \ref{th_2classes}, pencil \(\lambda E - Q\) cannot be regular, as \(\MM\) has a right minimal index larger than one. If \(\lambda E - Q\) is singular, then as a square pencil it must have both right and left minimal indices. In that case, by Theorem \ref{th_EQ}, \(\lambda E - Q\) must have a right minimal index zero, which corresponds to a common right nullspace of \(E\) and \(Q\). Hence, also \(E\) and \((J-R)Q\) have a common right nullspace. Therefore, pencil \(\MM\) has a right minimal index zero, so it is not strictly equivalent to \(\MN\), and we obtain a contradiction.
\end{example}

Consequently, we obtain the following observation.

\begin{observation}
    Let the pencil \(\MM = \lambda E - (J - R)Q\) be of the form \eqref{eq_form_general} with singular \(\lambda E - Q\). Then \(\MM\) has at least one right minimal index zero.
\end{observation}

Interestingly, there exist dH pencils that are also of the form \eqref{eq_form_general} with singular \(\lambda E - Q\).

\begin{example}
    Let
    \[
    \MM = \JJ_1^{\mu} \oplus \LL_0 \oplus \LL_0^T = \begin{bmatrix}
        \lambda - \mu & 0 \\
        0 & 0
    \end{bmatrix},
    \]
    where \(\mu = \alpha + i \beta\), with \(\alpha \leq 0\), \(\beta \in \RR\). We have \(\MM = \lambda E - (J - R) Q\) for
    \[
    E = \begin{bmatrix}
        1 & 0 \\
        0 & 0
    \end{bmatrix},
    \quad
    Q = \begin{bmatrix}
        1 & 0 \\
        0 & 1
    \end{bmatrix},
    \quad
    J = \begin{bmatrix}
        i \beta & 0 \\
        0 & 0
    \end{bmatrix},
    \quad
    \text{ and }
    R = \begin{bmatrix}
        - \alpha & 0 \\
        0 & 0
    \end{bmatrix},
    \]
    so that \(E^*Q = Q^*E \geq 0\), \(J^* = - J\), \(R^* = R \geq 0\), and \(\lambda E - Q\) is regular.
    On the other hand, we also have \(\MM = \lambda \hat{E} - (\hat{J} - \hat{R}) \hat{Q}\) for
    \[
    \hat{E} = \begin{bmatrix}
        1 & 0 \\
        0 & 0
    \end{bmatrix},
    \quad
    \hat{Q} = \begin{bmatrix}
        0 & 0 \\
        \mu & 0
    \end{bmatrix},
    \quad
    \hat{J} = \begin{bmatrix}
        0 & 1 \\
        -1 & 0
    \end{bmatrix},
    \quad
    \text{ and }
    \hat{R} = 0.
    \]
    Here, we obtain that \(\hat{E}^*\hat{Q}= \hat{Q}^*\hat{E} \geq 0\), \(\hat{J}^* = - \hat{J}\), and \(\hat{R}^* = \hat{R} \geq 0\), but \(\lambda \hat{E} - \hat{Q}\) is singular.
\end{example}

\section*{Conclusions and future research}

We have studied the strict equivalence orbit structure of dH pencils associated with linear time-invariant dissipative Hamiltonian descriptor systems. In particular, we derived a restriction on possible degenerations in a special case, characterised the maximal elements of the partial order induced by orbit closure inclusion, and analysed several properties of the orbit closures. A complete characterisation of the structure-preserving
stratification of dH pencils is left for future work. 

We also investigated pencils arising naturally in the study of orbit closures of dH pencils and obtained a partial characterisation of their Kronecker canonical form. Our results show that omitting an additional assumption fundamentally changes the admissible KCF structure: canonical blocks of every type and size may occur, although not in arbitrary combinations. A full characterisation of the KCF of pencils of this type remains an open problem.

\vspace{10pt}

{\small \textbf{Acknowledgements:}
The author would like to thank Micha\l{} Wojtylak for a valuable discussion on the subject.

The author gratefully acknowledges the support of the strategic programme Initiative for Excellence at the Jagiellonian University in Krak\'ow.
}

\end{document}